\documentclass[10pt,a4paper,reqno]{amsart}
\allowdisplaybreaks
\numberwithin{equation}{section}
\usepackage[normalem]{ulem}
\makeatletter 
\def\@tocline#1#2#3#4#5#6#7{\relax
  \ifnum #1>\c@tocdepth 
  \else
    \par \addpenalty\@secpenalty\addvspace{#2}%
    \begingroup \hyphenpenalty\@M
    \@ifempty{#4}{%
      \@tempdima\csname r@tocindent\number#1\endcsname\relax
    }{%
      \@tempdima#4\relax
    }%
    \parindent\z@ \leftskip#3\relax \advance\leftskip\@tempdima\relax
    \rightskip\@pnumwidth plus4em \parfillskip-\@pnumwidth
    #5\leavevmode\hskip-\@tempdima
      \ifcase #1
       \or\or \hskip 1em \or \hskip 2em \else \hskip 3em \fi%
      #6\nobreak\relax
    \dotfill\hbox to\@pnumwidth{\@tocpagenum{#7}}\par
    \nobreak
    \endgroup
  \fi}
\makeatother

\usepackage{transparent}
\usepackage[margin=1.3in]{geometry}
\usepackage{amsmath,amsthm,amssymb}
\usepackage{url}
\usepackage[dvipsnames]{xcolor}
\usepackage[english=american]{csquotes}
\usepackage{enumerate}
\usepackage{textcomp}
\usepackage{mathrsfs}
\usepackage{mathtools}
\usepackage{color, colortbl}
\definecolor{Gray}{gray}{0.9}
\usepackage{bbm}
\usepackage{stmaryrd}
\usepackage{hyperref}
\hypersetup{
    unicode=false,          
    pdftoolbar=true,        
    pdfmenubar=true,        
    pdffitwindow=false,     
    pdfstartview={FitH},    
    pdftitle={My title},    
    pdfauthor={Author},     
    pdfsubject={Subject},   
    pdfcreator={Creator},   
    pdfproducer={Producer}, 
    pdfkeywords={keyword1, key2, key3}, 
    pdfnewwindow=true,      
    colorlinks=true,       
    linkcolor=blue ,          
    citecolor=blue ,        
    filecolor=magenta,      
    urlcolor=Aquamarine           
}
\usepackage{xcolor}
\usepackage{epsfig,color,graphicx,graphics}
\usepackage{caption}
\usepackage{amsfonts}
\usepackage{tikz}
\usepackage{siunitx}
\usepackage{booktabs,colortbl, array}
\usepackage{pgfplotstable}
\pgfplotsset{compat=1.8}

\definecolor{rulecolor}{RGB}{0,71,171}
\definecolor{tableheadcolor}{gray}{0.92}

\newtheorem{theorem}{Theorem}[section]
\newtheorem{lemma}[theorem]{Lemma}
\newtheorem{proposition}[theorem]{Proposition}
\newtheorem{corollary}[theorem]{Corollary}

\theoremstyle{definition}
\newtheorem{definition}[theorem]{Definition}
\newtheorem{remark}[theorem]{Remark}

\newtheorem{assumption}[theorem]{Assumption} 

\newcommand{\eye}{\mathbf{1}}
\newcommand{\sgn}{\mathrm{sgn}}

\newcommand{\Langle}{\left \langle}
\newcommand{\Rangle}{\right \rangle}

\newcommand{\cF}{\mathcal{F}}
\newcommand{\cD}{\mathcal{D}}
\newcommand{\cM}{\mathcal{M}}
\newcommand{\cH}{\mathcal{H}}

\newcommand{\Ccal}{\mathcal{C}}
\newcommand{\Dcal}{\mathcal{D}}

\newcommand{\Fcal}{\mathcal{F}}

\newcommand{\Hcal}{\mathcal{H}}

\newcommand{\Mcal}{\mathcal{M}}

\newcommand{\Sbb}{\mathbb{S}}

\DeclareMathOperator{\id}{id}

\DeclareMathOperator{\diam}{diam}

\DeclareMathOperator*{\wslim}{w*-lim}

\DeclareMathOperator{\diverg}{div}

\DeclareMathOperator{\dist}{dist}

\DeclareMathOperator{\Tan}{Tan}

\DeclareMathOperator{\supp}{supp}

\newcommand{\N}{\mathbb{N}}
\newcommand{\R}{\mathbb{R}}

\newcommand{\loc}{\mathrm{loc}}

\newcommand{\spt}{\mathrm{spt} \,}

\newcommand{\toweakstar}{\overset{*}\rightharpoonup}

\newcommand{\eps}{\varepsilon}
\newcommand{\vphi}{\varphi}

\makeatletter
\renewcommand*\env@matrix[1][*\c@MaxMatrixCols c]{%
    \hskip -\arraycolsep
    \let\@ifnextchar\new@ifnextchar
    \array{#1}}
\makeatother

\DeclareMathOperator{\Lip}{Lip}

\newcommand{\mres}{\mathbin{\vrule height 1.6ex depth 0pt width
        0.13ex\vrule height 0.13ex depth 0pt width 1.3ex}}
\newcommand{\restr}{\mathbin{\vrule height 1.6ex depth 0pt width
        0.13ex\vrule height 0.13ex depth 0pt width 1.3ex}}

\DeclareMathOperator{\Id}{id}

\newcommand{\divr}{\mathrm{div}}

\def\vint_#1{\mathchoice%
    {\mathop{\kern 0.2em\vrule width 0.6em height 0.69678ex depth -0.58065ex
            \kern -0.8em \intop}\nolimits_{\kern -0.4em#1}}%
    {\mathop{\kern 0.1em\vrule width 0.5em height 0.69678ex depth -0.60387ex
            \kern -0.6em \intop}\nolimits_{#1}}%
    {\mathop{\kern 0.1em\vrule width 0.5em height 0.69678ex depth -0.60387ex
            \kern -0.6em \intop}\nolimits_{#1}}%
    {\mathop{\kern 0.1em\vrule width 0.5em height 0.69678ex depth -0.60387ex
            \kern -0.6em \intop}\nolimits_{#1}}}
\makeatletter
\newcommand*{\RangeX}{%
    {%
        \mathpalette\@RangeOf{X}%
    }%
}
\newcommand*{\@RangeOf}[2]{%
    \sbox0{$\m@th#1\mathsf{#2}$}%
    \mathsf{#2}%
    \kern-\wd0 %
    \mkern2.75mu\relax
    \nonscript\mkern.25mu\relax
    \mathsf{#2}%
}
\makeatother

\newcommand{\aveint}[2]{\mathchoice%
    {\mathop{\kern 0.2em\vrule width 0.6em height 0.69678ex depth -0.58065ex
            \kern -0.8em \intop}\nolimits_{\kern -0.45em#1}^{#2}}%
    {\mathop{\kern 0.1em\vrule width 0.5em height 0.69678ex depth -0.60387ex
            \kern -0.6em \intop}\nolimits_{#1}^{#2}}%
    {\mathop{\kern 0.1em\vrule width 0.5em height 0.69678ex depth -0.60387ex
            \kern -0.6em \intop}\nolimits_{#1}^{#2}}%
    {\mathop{\kern 0.1em\vrule width 0.5em height 0.69678ex depth -0.60387ex
            \kern -0.6em \intop}\nolimits_{#1}^{#2}}}

\title{On the geometric properties of multi-operator two-phase elliptic measure}

\author[M. Goering]{Max Goering}
\address{Affiliation: Department of Mathematics and Statistics, University of Jyväskylä, Finland. \newline
\indent Address: Max Goering, Department of Mathematics and Computer Science, Technical University of Eindhoven,   5600 MB Eindhoven, The Netherlands}
\email{mlgoering@gmail.com}

\author[A. Skorobogatova]{Anna Skorobogatova}
\address{Anna Skorobogatova, Institute for Theoretical Studies, ETH Z\"{u}rich, Scheuchzerstrasse 70
8092 Z\"{urich}, Switzerland}
\email{anna.skorobogatova@eth-its.ethz.ch}

\begin{document}

\begin{abstract}
    We provide a structural characterization of a given boundary using two-phase elliptic measure in a multi-operator setting, extending to this {novel} setting results of Kenig, Preiss \& Toro, Toro \& Zhao and Azzam \& Mourgoglou, including {the validity of Oksendal's conjecture for two-sided NTA domains} under the assumption of mutual absolute continuity of the elliptic measures. Our techniques rely on a reduction to a multi-operator two-phase free-boundary problem combined with an extension of the powerful tools introduced by Preiss in his Density Theorem.
\end{abstract}

\maketitle

\tableofcontents

\section{Intro}
\subsection{The multi-operator setting and new results}
In this work, we study the structure of the points of mutual absolute continuity for a pair of elliptic measures $\omega^{\pm}$ corresponding to two \emph{different} divergence-form operators on complementary NTA domains $\Omega^{+}$ and $\Omega^- = \R^n \setminus \overline{\Omega^+}$. Identifying the relationship between the structure of a boundary and the behavior of its elliptic measure has a long history, starting with the classical case of harmonic measure, i.e. when the operator in question is the Laplacian. We provide an overview of the history of the problem in Section \ref{ss:background} below. Our setting herein falls into the regime of \emph{two-phase} elliptic measure problems. In contrast to the one-phase setting, one does not relate the behavior of the elliptic measure and surface measure of the boundary, but instead considers the relative behavior between the two elliptic measures for the complementary domains.

More precisely, our setting is as follows. Consider complementary domains $\Omega^+ \subset \R^n$ and  $\Omega^- = \R^n \setminus \overline{\Omega^+}$, with $n\geq 3$. Suppose that $\Omega^+$ is a two-sided local NTA  domain (see Section \ref{s:prelim}). Recall that\footnote{See, for instance, \cite{KPT}.} if $\Omega^{\pm}$ are two-sided NTA with $\partial \Omega^+ = \partial \Omega^- =: \partial \Omega$ then the Dirichlet problem for an elliptic divergence-form operator with any continuous boundary data admits a solution in $W^{1,2}_{\loc}(\Omega^{+})\cap C(\overline{\Omega^+})$, and the elliptic measure associated to this Dirichlet problem is well-defined. For each choice of sign, let $L_{A^{\pm}} u = - \diverg(A^\pm \nabla u)$ be the second order linear elliptic differential operator with symmetric and uniformly elliptic coefficients $A^{\pm}$, where $A^{\pm}$ are $2$-quasicontinuous matrix-valued functions on $\R^{n}$. Recall that a function $f: \R^{n} \to \R$ is called $2$-quasicontinuous if for all $\eps > 0$ there exists an open set $V$ so that $\textrm{cap}_{2}(V) \le \eps$ and $f|_{\R^{n} \setminus V}$ is continuous, see \cite[Sections 4.7, 4.8]{EvansGariepy}. 

\begin{remark}
    We work with $2$-quasicontinuous functions $A^\pm : \R^{n} \to \R^{n \times n}$ precisely so that for $\omega^{\pm}$-a.e. $p$ we know $\lim_{r \to 0} \frac{1}{\omega^{\pm}(B(p,r))} \int |A^\pm - A^\pm(p)| d \omega^{\pm} = 0$, as soon as $\Omega^{\pm}$ are Wiener regular. See \cite[Theorem 11.14]{HKT} for more details. 
\end{remark}

\begin{remark}
    Since we expect tools from complex analysis will provide stronger answers with simpler techniques to the questions answered herein, we do not consider the case $n=2$ even though the results remain valid in that case under the assumption that the coefficients are continuous outside a set of zero logarithmic capacity (see for instance the difference between Theorem \ref{t:planardecomp} and Theorem \ref{t:kptdecomp}).
\end{remark}

Whenever $A^{\pm}$ and $\Omega^{\pm}$ are understood, we allow $\omega^{\pm}$ to denote the elliptic measures with suppressed poles $x^{\pm} \in \Omega^{\pm}$ (see Section \ref{s:prelim} for more details).

Recall that a non-zero Radon measure $\nu$ is a \emph{tangent measure} to $\mu$ at $a\in \R^n$, denoted $\nu \in \Tan(\mu,a)$, if there exists $c_{i} > 0$ and $r_{i} \downarrow 0$ so that the pushforwards $T_{a,r_{i}}[\mu]$ of $\mu$ under the maps $T_{a,r_i}(x) := \frac{x-a}{r_i}$ satisfy $\lim_{i \to \infty} c_{i} T_{a,r_{i}}[\mu] = \nu$ in the weak-$*$ sense \cite{Preiss}. We denote the space of $(n-1)$-dimensional flat measures in $\R^{n}$ by
    \begin{equation}\label{e:flat-meas}
        \cF = \{ c \cH^{n-1} \restr \pi \mid \pi \in G(n-1,n) \},
    \end{equation}
    where $G(m,n)$ denotes the Grassmannian of $m$-dimensional linear subspaces of $\R^n$.

Our main result is the following.

\begin{theorem}\label{t:decomp}
    Suppose that $\Omega^\pm$ {are complementary NTA domains with common boundary} $\partial\Omega$, {that $A^{\pm}:\R^{n} \to \R^{n \times n}$ are $2$-quasicontinuous and uniformly elliptic matrix-valued functions}, and that $L_A^{\pm}{= - \divr \left( A^{\pm} \nabla \cdot \right)}$ {on $\Omega^{\pm}$}. If
    \[
        F_{1} = \left\{y \in \partial \Omega : 0 < \frac{d\omega^-}{d\omega^+}(y) < + \infty\right\}\,,
    \]
    then there exists {an $\omega^
\pm$-}full measure subset of $F_{1}$, denoted $F^{*}$, so that 
    $\partial \Omega = F^{*} \sqcup S \sqcup N$
     and 
    \begin{itemize}
        \item[(i)] {$\Tan(\omega^{\pm},p) \subset \cF$ for all $p \in F^{*}$}, $\dim_\Hcal(F^*) {= n-1}$, and on $F^{*}$, $\omega^{+} \ll \omega^{-} \ll \omega^{+}$;
        \item[(ii)] $\omega^+ \perp \omega^-$ on $S$;
        \item[(iii)] $\omega^\pm(N) = 0$.
    \end{itemize}
\end{theorem}

Recent results \cite{david2021good,perstneva2025good} demonstrate that in the one-phase elliptic measure setting, sufficient regularity of the elliptic coefficients are necessary at the boundary for the relationship between geometry of a boundary and absolute continuity of the elliptic measure to hold. {Note that although in \cite{david2021good}, the coefficients of the operator are continuous up to the boundary, the setting therein is not a two-phase one since the boundary is a four-corner Cantor set. On the other hand, the snowflake examples constructed in \cite{perstneva2025good}, which are admissible in the two-phase setting, \emph{do not} have coefficients that are continuous up to the boundary; the oscillations degenerate as one approaches the boundary, see Remark 2 therein. In fact, a by-product of our result is that no such example can be constructed with coefficients that are 2-quasicontinuous, thus demonstrating an obstruction to how badly behaved the coefficients can be at the boundary in any example of this type.}

\subsection{Motivating background and history of harmonic measure}\label{ss:background}

{
The motivation for our study of the relationship between regularity properties of the two-phase multi-operator elliptic measures and structure of the boundary stem from powerful analogous results for two-phase harmonic measure, that is, the case where $A^\pm = \Id$. Let us briefly describe the history of this problem, starting with the story in the plane. The following is a combination of results due to Makarov, McMillan, Pomerenke, and Choi; see \cite{garnett2005harmonic} for the precise references.

\begin{theorem} \label{t:planardecomp}
    If $\partial \Omega\subset \R^2$ is a Jordan curve and $\omega$ the corresponding harmonic measure, then one can write $\partial \Omega$ as a disjoint union $\partial\Omega = G \sqcup S \sqcup N$, with
    \begin{itemize}
        \item[(1)] $\omega(N) = 0$;
        \item[(2)] $\cH^{1}(S) = 0$;
        \item[(3)] $\limsup_{r \to 0} r^{-1} \omega(B(p,r)) = + \infty$ and $\liminf_{r \to 0} r^{-1} \omega(B(p,r)) = 0$ for $\omega$-a.e. $p \in S$;
        \item[(4)] there is a geometric description of the points in $S$ as twist points\footnote{See \cite{garnett2005harmonic} for a definition of twist points.};
        \item[(5)] $\omega \ll \cH^{1} \ll \omega$ on $G$;
        \item[(6)] Every point in $G$ is a cone point for $\Omega$, that is, the vertex of a cone contained in $\Omega$;
        \item[(7)] $\omega$- a.e. and $\cH^{1}$- a.e. cone point for $\Omega$ is in $G$
        \item[(8)]  For $\omega$-a.e. $p \in G$, $\lim_{r \to 0} r^{-1} \omega(B(p,r))$ exists and takes values in $(0,\infty)$.     
    \end{itemize}
\end{theorem}
Note that property (8) implies that $\omega \restr G$ is rectifiable by Preiss' Density Theorem \cite{Preiss}.

We recall the Hausdorff dimension of $\omega$, denote $\dim_{\cH}(\omega)$ is defined by
$$
\dim_{\cH}(\omega) = \inf\{s : \exists E \subset \partial \Omega \text{ with } \cH^{s}(E) = 0 \text{ and } \omega(E \cap K) = \omega(\partial \Omega \cap K) \, \forall \text{ compact } K \subset \R^{n}\}.
$$
Through a slight abuse of notation, we will also write $\dim_{\cH}(E)$ to denote the Hausdorff dimension of a set $E$.

In the plane, Makarov established that for simply connected domains, $\dim_{\cH}(\omega) = 1$ \cite{makarov1985distortion}, answering Oskendal's Conjecture in the plane, concerning which type of domains $\dim_{\cH}(\omega) = 1$ for. For general domains in $\R^{2}$, it is only known that $\dim_{\cH}(\omega) \le 1$, \cite{carleson1985support}, \cite{jones1988hausdorff}. 

In $\R^{n}$ for $n \ge 3$, Bourgain \cite{bourgain1987hausdorff} showed there exists a constant $\beta_{n} > 0$ so that $\dim_{\cH}(\omega) \le n - \beta_{n}$. In \cite{wolff1995counterexamples}, Wolff showed that Oksendal's conjecture fails in general for $n \ge 3$ by constructing ``Wolff snowflakes" for which one can have $\dim_{\cH}(\omega) < n-1$ and $\dim_{\cH}(\omega) > n-1$. Wolff's examples are two-sided NTA domains. Lewis, Verchota \& Vogel \cite{LVV} re-examined Wolff's snowflakes and discovered one can simultaneously have $\dim_{\cH}(\omega^{\pm}) < n-1$ or $\dim_{\cH}(\omega^{\pm}) > n-1$. However, due to the Alt-Caffarelli-Friedman monotonicity formula, if $\omega^{+} \ll \omega^{-} \ll \omega^{+}$ then $\dim_{\cH}(\omega^{\pm}) \ge n-1$. This final observation led Bishop \cite{bishop1992questions} to ask whether $\omega^{\pm}(E) > 0$ and $\omega^{+} \ll \omega^{-} \ll \omega^{+}$ ensure that {$\omega^\pm$ are both mutually absolutely continuous with respect to $\Hcal^{n-1}$ on $E$, modulo a set of $\omega^\pm$-measure zero, and thus} $\dim_{\cH}(E) = n-1$, where $\dim_{\cH}(E)$ denotes the Hausdorff dimension of $E$. Bishop's question was answered positively {in \cite{azzam2017mutual} for CDC domains, and subsequently in \cite{azzam2019two} in full generality.}

In the absence of complex analytic tools, studying the geometry of $\partial \Omega$ for rough domains $\Omega\subset \R^n$ for $n\geq 3$ in terms of the behavior of harmonic measure relies heavily on ``two-phase" techniques for the harmonic measure. The following slightly weaker version of a boundary decomposition theorem holds for harmonic measure in $\R^{n}$, $n\geq 3$.

\begin{theorem}[\cite{KPT}] \label{t:kptdecomp}
    Let $n \ge 3$. If $\Omega^{\pm} \subset \R^{n}$ are complementary NTA domains and $\Omega^{\pm}$ are the interior and exterior harmonic measures, then $\partial \Omega = \Gamma \sqcup S \sqcup N$ where
    \begin{itemize}{
        \item[(1)] $\omega^{+} \perp \omega^{-}$ on $S$ and $\omega^{\pm}(N) = 0$.
        \item[(2)] On $\Gamma$, $\omega^{+} \ll \omega^{-} \ll \omega^{+}$ and $\dim_{\cH} \Gamma \le n-1$, and if $\omega^\pm(\Gamma)>0$ then $\dim_{\cH} \Gamma = n-1$.
        \item[(3)] If $\cH^{n-1} \restr \partial \Omega$ is a Radon measure then $\Gamma$ is $(n-1)$-rectifiable and $\omega^{\pm} \ll \cH^{n-1} \ll \omega^{\pm}$ on $\Gamma$
        \item[(4)] For every point $p \in \Gamma$, $\Tan(\omega^{\pm},p) \subset \cF$.}
    \end{itemize}
\end{theorem}

The conclusion of Theorem \ref{t:kptdecomp} was extended to the case of two-phase elliptic measure \cite{AzzMou} for a \emph{single} elliptic operator whose coefficients have sufficient regularity across the boundary $\partial\Omega$. One can thus view our main result Theorem \ref{t:decomp} as an answer to the question: does the above relationship between the regularity of interior- and exterior- elliptic measure and the geometry of the boundary remain true in the \emph{multi-operator} setting, where the coefficients of the corresponding elliptic operators may have a discontinuity across the boundary? 

In particular, in the multi-operator setting, we {establish the validity Oksendal's conjecture for NTA domains} under the assumption of mutual absolute continuity of the pair of elliptic measures. Namely, we show that the set of points of mutual absolute continuity has dimension {$n-1$ (modulo a set of $\omega^\pm$-measure zero)} and we make progress toward a geometric characterization by showing that on a full measure subset of these points the elliptic measures $\omega_{A^{\pm}}^{\pm}$ have flat tangents. {To show Oksendal's conjecture in the present setting, we invoke an argument in \cite{AzzMou} in place of that in \cite{KPT}, which avoids the use of the ACF monotonicity formula.} Although an analogue of such a monotonicity formula exists in the multi-operator setting (see \cite{TerrSoa}), it does not share the property of characterizing flat regions of the boundary as in the single-operator case, thus preventing it from being useful in the study of the regularity properties of the boundary; cf. Remark \ref{r:TerrSoa}.

Two notable differences between Theorem \ref{t:planardecomp} and Theorem \ref{t:kptdecomp} are that, first, mutual absolute continuity of the harmonic and surface measure on the ``good set" is established in the plane and not in Theorem \ref{t:kptdecomp} without any additional assumption and, second, only in the plane is there a geometric characterization of the ``good set" and the ``singular set" {in terms of the surface measure}.  Both the mutual absolute continuity of harmonic and surface measure and a geometric characterization of the ``good set" are handled in the remarkable works \cite{azzam2017mutual,azzam2019two}, which can be summarized as follows.

\begin{theorem} \cite{azzam2017mutual,azzam2019two} \label{t:amtdecomp}
    Let $n \ge 3$. If $\Omega^{\pm} \subset \R^{n}$ are complementary domains, then $\partial \Omega = G \sqcup S \sqcup N$, where
    \begin{itemize}
        \item[(1)] $\omega^{+} \perp \omega^{-}$ on $S$ and $\omega^{\pm}(N) = 0$
        \item[(2)] $G$ is $(n-1)$-rectifiable and $\omega^{\pm} \ll \cH^{n-1} \ll \omega^{\pm}$ on $G$.
    \end{itemize}
\end{theorem}

Note that a characterization of $S$ in terms of surface measure is not possible in general without additional assumptions such as uniform upper $(n-1)$ density control (see \cite{Badger12}), as demonstrated via a counterexample in \cite{AMT-counterex}.

The proof of Theorem \ref{t:amtdecomp} relies heavily on the Riesz transform (in particular the results of \cite{girela2018riesz}). Thus, the analogous rectifiable structure (up to an $\omega^\pm$-null set) for mutually absolutely continuous points is currently out of reach in the setting herein. We conjecture that such a strengthened characterization should also hold true here, however.

\subsection{Overview of proof of Theorem \ref{t:decomp} and structure of article}\label{s:overview}

The following class of measures identifies the free-boundary problem that our tangent measures satisfy at a.e. point in $F_{1}$.  

\begin{definition} \label{d:cda1a2}
For any two elliptic matrices $A^{\pm}$, we write $\nu \in \cD(A^{+},A^{-})$ to mean that $\nu$ is a Radon measure with $0 \in \spt \nu$ that satisfies the following properties:
\begin{enumerate}
    \item $\R^{n} \setminus \spt ~ \nu = \Omega^{+} \cup \Omega^{-}$ for unbounded complementary NTA domains $\Omega^\pm$ with fixed NTA constants.
    \item There exist non-negative functions $u^{\pm} >0$ on $\Omega^{\pm}$ and supported on $\overline{\Omega^{\pm}}$ respectively, so that
    $$
        \int \varphi \, d \nu = \int_{\R^{n}} \nabla u^{\pm} \cdot (A^{\pm} \nabla \varphi) \qquad \forall \varphi \in C^{\infty}_{c}(\R^{n}).
    $$

    The functions $u^{\pm}$ are referred to as Green's functions with pole at infinity for the triple $(\Omega^{\pm},\nu,L_{A^{\pm}})$, see Remark \ref{r:poleatinf}.
\end{enumerate}
\end{definition}

{Definition \ref{d:cda1a2} can be viewed as the property that $\nu$ is supported on an NTA boundary admitting a pair of Greens functions for $L_{A^\pm}$-elliptic measures with poles at infinity on either side.} We have the following theorem (see Theorem \ref{t:blowups} for a more detailed version).
\begin{theorem} \label{t:maintangents}
    Suppose that $\Omega^\pm$, $L_A$ and $\omega^\pm$ are as above. For $\omega^{\pm}$-a.e. point $p$ in the set $F_1$ defined in Theorem \ref{t:decomp},
    $$
        \Tan(\omega^\pm,p) \subset \cD(A^+(p), A^{-}(p)).
    $$
\end{theorem}

The proof of Theorem \ref{t:decomp}, much like its analogues in the case of harmonic measure \cite{KPT} or a single elliptic operator \cite{AzzMou}, relies on a procedure first introduced by Preiss in his resolution of the density theorem. We will now briefly summarize the procedure of Preiss and how to use tangent measures, but with the specific focus on proving the claim that $\Tan(\omega^{\pm},p) \subset \cF$ in Theorem \ref{t:decomp}. Therefore, this summary also provides a broad outline of the rest of the paper. We wish to apply a key lemma regarding the connectedness of the space of tangent measures, Lemma \ref{l:connectedness}, in the setting where $\cM = \cD(A_{1}, A_{2})$. We encourage the reader to look at that lemma before reading the following three-step overview of what remains:
\begin{enumerate}
    \item For $\omega^{\pm}$-a.e. $p \in F_{1}$, there are elliptic matrices $A_{1}, A_{2}$ so that $\Tan(\omega^{\pm},p) \subset \cD(A_{1}, A_{2})$; this is the contents of Theorem \ref{t:blowups}. {In fact, we verify that this is equivalent to the conclusion that for $\Lambda(p) = \sqrt{A_{1}}$, the \emph{$\Lambda$-tangents} for $\omega^\pm$-a.e. point $p \in F_1$ satisfy $\Tan_{\Lambda}(\omega^{\pm},p) \subset \cD(\id, \Lambda(p)^{-1} A_{2} \Lambda(p))$.}
    \item The pair $(\cF, \cD(\id, \Lambda(p)^{-1} A_{2} \Lambda(p)))$ satisfy \eqref{e:pi}, {which we demonstrate guarantees a rigidity property for Radon measures in the class $\cD(\id, \Lambda(p)^{-1} A_{2} \Lambda(p))$ that are flat at infinity}. This is the hardest step, and 
    is shown in Section \ref{s:freeboundary}. {In particular, in Sections \ref{ss:viscosity-sol} and \ref{ss:eps-mon}, we demonstrate that after a suitable blow-up procedure at $\omega^{\pm}$-a.e. $p \in F_{1}$, the difference of limiting Greens functions is a viscosity solution of a free-boundary problem with both a Dirichlet and a transmission condition at the free-boundary, the latter agreeing with the blown up boundary at $p$. We are then able to exploit known regularity and rigidity techniques for such free boundary problems.}
    \item After the first two steps we can apply Lemma \ref{l:connectedness}, but risk the case that the dichotomy tells us we have no flat tangents. We resolve this concern by applying \cite[Proposition 4.4]{AM} (which we restate with additional details in Lemma \ref{l:twoplane}) and Theorem \ref{t:tan2ltan} to conclude that for $\omega^{\pm}$-a.e. $p$, $\Tan_{\Lambda}(\mu,p) \cap \cF \neq \emptyset$, thus ruling out case (ii) in the dichotomy and proving the claim that for $\omega^{\pm}$-a.e. $p \in F_{1}$, $\Tan(\omega^{\pm},p) \subset \cF$.
\end{enumerate}

The experienced reader may notice that we passed between tangent measures and $\Lambda$-tangents in our outline of the paper above. This is done formally via Lemma \ref{l:iso} and is useful to perform an appropriate linear change of variables to simplify the structure of our free boundary problem to match the form as that previously studied in \cite{Feldman, AM, Cerutti-Ferrari-Salsa, Ferrari-Salsa, CSS-2d}.

Step (2) above involves demonstrating that if a blown up boundary is sufficiently close to flat, then it is \emph{exactly} flat. In the case $A^{\pm} =\Id$, the free boundary problem reduces significantly to the case where the difference $u = u^{+} - u^{-}$ of the blown up Greens functions $u^\pm$ is a harmonic polynomial. The latter was the crucial property exploited in \cite{KPT, AzzMou} (see also \cite{BET}) in order to conclude the desired rigidity result, which follows from a Liouville-type theorem for harmonic polynomials. Here, however, the coefficients $A^\pm(p)$ differ on each side of $\partial \Omega$ and thus one cannot generally hope to realize the blown-up boundary as the nodal set of an elliptic polynomial. We instead replace this with the delicate study of the aforementioned multi-operator free boundary problem, for which we can obtain an analogous Liouville-type, adapting ground-breaking techniques of Caffarelli \cite{Caff1-Lip-implies-reg,Caff2-flat-implies-Lip,Caff3}, which had {already been extended to a setting with two different operators in} \cite{Feldman,AM,Cerutti-Ferrari-Salsa,Ferrari-Salsa}. {The crucial observation herein is that these results may be used here, since our blowups fall into the very same framework.}

The overall idea is a two-step procedure: first show that flatness of the boundary implies the boundary is Lipschitz, and then prove that Lipschitz boundary regularity in turn implies $C^{1,\alpha}$ boundary regularity. Although this is a local argument, which has since been superseded by the more flexible approach of De Silva \cite{DeSilva}, we crucially rely on the global rigidity consequence of this, {which seems out of reach with De Silva's techniques}. We refer the reader to Section \ref{ss:eps-mon-implies-Lip} for a more detailed discussion of the approach. Note that in the special case where the jump in the coefficients is characterized by a scalar function, namely $A^+(p) = c A^{-}(p)$, one may consider a suitably weighted difference of $u^+$ and $u^-$ in order to once again reduce to the elliptic polynomial case; see \cite{KLS1,KLS2}. In that setting, the same procedure as in \cite{AzzMou} carries through without a hitch, leaving it uninteresting in our setting. In particular, \cite{KLS1,KLS2} make use of the Alt-Caffarelli-Friedman monotonicity formula and Almgren's frequency function, neither of which we can use in this setting.

\subsection*{Acknowledgments}
The authors wish to thank Tatiana Toro for introducing this problem to them, hosting them at SLMath (formerly MSRI) and for a number of valuable discussions and insights. They also thank Max Engelstein, Dennis Kriventsov, and Steve Hofmann for fruitful discussions, as well as all of the staff at SLMath for their hospitality. We thank the anonymous referees for their helpful feedback, which allowed us to improve the presentation of the article greatly.

A.S. is grateful for the generous support of Dr.~ Max R\"ossler, the Walter Haefner Foundation and the ETH Z\"urich Foundation. This research was partially conducted during the period A.S. served as a Clay Research Fellow.

M.G. is supported by the European Research Council (ERC) under the European Union's Horizon Europe research and innovation programme (grant agreement No 101087499). 

Part of this work was funded by the grant DMS-1853993.

\section{Preliminaries}\label{s:prelim}

\subsection{Notation and conventions}

Unless otherwise stated, balls $B(x,r)$ of radius $r$ and center $x$ are assumed to be open, and $n \geq 3$ is a natural number. When the center of a given ball $B(0,r)$ is the origin, we instead often write $B_r$. A \emph{domain} is an open connected set, and $p$ denotes a point in the boundary of a given domain.  We use $\Omega, \Omega^{+}, \Omega^{-}$ to denote domains. We assume $\Omega^{\pm}$ are complementary domains, that is $\Omega^{-} = \R^{n} \setminus \overline{\Omega^{+}}$. We will often denote ``boundary balls" $B(p,r) \cap \partial\Omega$ by $\Delta(p,r)$.

We denote by $\delta_{\Omega}(x)$ the distance from $x$ to $\partial \Omega$. When the domain is understood, we simply write $\delta(x)$. 

We refer to $m$-dimensional linear subspaces of $\R^n$ as $m$-planes, and $G(m,n)$ denotes the space of $m$-planes in $\R^{n}$. We will be taking $m=n-1$ throughout, and our $(n-1)$-planes will be often denoted by $\pi$. We denote the pushforward of a Radon measure $\mu$ under a Borel map $T$ by $T[\mu]$.

\subsection{NTA domains}
We will work under the following assumptions on our domains $\Omega$, $\Omega^+$ and $\Omega^-$ throughout, as introduced in \cite{JK}. The first is a quantitative openness known as the \emph{corkscrew condition}.

\begin{definition}[Corkscrew Condition] \label{d:csc}
    Let $C_0, R_0>0$. We say that an open set $\Omega \subset \R^{n}$ satisfies the $(C_0,R_{0})$ interior corkscrew condition if for every $p \in \partial \Omega$ and $r \in (0,R_{0})$ there exists a point $x_{p,r}$, called the interior corkscrew point for $p$ at scale $r$, so that $B(x_{p,r}, r/C_0) \subset B(p,r) \cap \Omega$.

    We say that an open set $\Omega \subset \R^{n}$ satisfies the $(C_0,R_{0})$ two-sided corkscrew condition if both $\Omega$ and $\R^n\setminus \overline{\Omega}$ satisfy the $(C_0,R_{0})$ interior corkscrew condition.
\end{definition}

The next assumption is a quantitative notion of connectedness known as the \emph{Harnack chain condition}.

\begin{definition}[Harnack Chain Condition] \label{d:hcc}
    We say that a domain $\Omega$ satisfies the $(C_1,R_{1})$-Harnack chain condition if for every $0 < \rho \le R_1$, $\Lambda \ge 1$, and every pair of points $x,x^{\prime} \in \Omega$ with $\min\{\delta_{\Omega}(x), \delta_{\Omega}(x^{\prime})\} \ge \rho$ and $|x-x^{\prime}| < \Lambda \rho$ there is a chain of balls $B^{1}, \dots, B^{N} \subset \Omega$ with $N \le C_1(\log_{2} \Lambda + 1)$ so that:
    \begin{enumerate}
        \item $x \in B^{1}$, $x^{\prime} \in B_{N}$
        \item $B^{k} \cap B^{k+1} \neq \emptyset$ for all $k=1, \dots, N-1$
        \item $C_1^{-1} \diam (B^{k}) \le \dist_\Hcal[B^{k}, \partial \Omega] \le C_1 \diam (B^{k})$,
    \end{enumerate}
    where $\dist_\Hcal$ denotes the Hausdorff distance. The chain of balls is called a \emph{Harnack chain}.
\end{definition}

\begin{definition}[NTA domains] \label{d:NTA}
We say that an open set $\Omega \subset \R^{n}$ is a non-tangentially accessible domain (NTA domain) with constants $(C_2,R_{2})$ if it satisfies the $(C_2,R_{2})$-Harnack chain condition and the $(C_2,R_{2})$ two-sided corkscrew condition.

When we discuss an NTA domain, we will typically neglect to specifically mention the associated constants $(C_2,R_{2})$, which are referred to as the \emph{NTA constants}, since they will always be fixed. If $\Omega$ is unbounded, we require that $\R^{n} \setminus \partial \Omega$ consists of two non-empty connected components. If $\Omega$ is unbounded, then $R_{0}= \infty$ is allowed.

Sometimes authors make the distinction between one-sided and two-sided NTA domains. In such cases, our definition coincides with two-sided NTA domains.

\end{definition}

Let $\dist_{\Hcal}(A,B)$ denote the Hausdorff distance between two non-empty compact sets $A,B \subset \R^{n}$, that is, for non-empty compact sets $A,B \subset \R^{n}$,
 $$
 \dist_{\Hcal}[A,B] := \sup \{d(a,B) : a \in A \} + \sup \{d(b,A) : b \in B \}.
 $$
 For a closed set $\Gamma \subset \R^{n}$, for any $p \in \Gamma$ and any $r > 0$, we define

\begin{equation} \label{e:Theta}
\Theta_{\Gamma}(p,r) = \frac{1}{r}\inf \left\{ \dist_{\Hcal}\Big[(\Gamma -p)\cap \overline{B_{r}}, \pi \cap \overline{B_{r}}\Big] : \pi \in G(n-1,n)\right\}\,   
\end{equation}
and
\begin{equation} \label{e:beta}
\beta_{\Gamma}(p,r) = \frac{1}{r} \inf_{\pi \in G(n-1,n)} \sup_{q \in \Gamma \cap B(p,r)} \dist(q, p + \pi)
\end{equation}

 We recall the following lemma from \cite{KPT}, which will be useful to obtain the dimension estimate on $F^*$ in Theorem \ref{t:decomp}.

\begin{lemma}[{\cite[Lemma 2.4]{KPT}}]\label{l:beta-decay-dim}
     Suppose that $\Gamma\subset \R^n$ is a closed set with the property that
     \[
        \lim_{r\to 0} \beta_\Gamma(p,r) = 0 \qquad \forall \ p \in \Gamma\,.
     \]
     Then $\dim_\Hcal(\Gamma) \leq n-1$\,.
 \end{lemma}

\subsection{Elliptic measure background} \label{s:emb}

Suppose that $A \in L^{\infty}(\Omega; \R^{n\times n})$ is a $\Lambda_{0}$-uniformly elliptic symmetric matrix for some $\Lambda_0 \geq 1$, that is $A^{t} = A$ and
$$
\Lambda_{0}^{-1} |\xi|^{2} \le \langle \xi, A \xi \rangle \le \Lambda_{0} |\xi|^{2} \qquad \text{for every $\xi \in \R^n\setminus \{0\}$.}
$$

If $\Omega$ is an NTA domain (see \ref{d:NTA}) then in particular $\Omega$ is Wiener regular (see \cite{Wie}) and so the work \cite{LSW} guarantees that it is also regular for $L_A = - \divr(A \nabla \cdot )$ and so the elliptic measure for $(\Omega,L_{A})$ with pole at $x$ is well-defined. That is, for $x \in \Omega$ there is a unique probability measure $\omega^{x}$ supported on $\partial\Omega$ so that if {$f \in C(\partial \Omega)$ and $u \in W^{1,2}(\Omega)$ weakly solves $L_{A} u = 0$ with boundary data $f$, then $u(x) = \int f d \omega^{x}$.} 

Note that when discussing regularity properties of $\omega^{\pm}$ up to an $\omega^\pm$-negligible set, we may without loss of generality fix such poles $x^\pm$ arbitrarily, so long as $\Omega^{\pm}$ are Wiener regular, since any pair of elliptic measures in a fixed domain associated to a given uniformly elliptic operator are mutually absolutely continuous, due to the Harnack inequality. {Indeed, suppose that $\Omega$ is Wiener regular, let $x,y \in \Omega$ and suppose that $K$ is a compact set such that $\omega^{x}(K) = 0$. Then for all $\eps > 0$ there exists an open neighborhood $U_{\eps}$ of $K$ with $\omega^{x}(U_{\eps}) < \eps$. Consider the functions $f_{\eps} \in C(\partial \Omega)$ so that $\eye_{K} \le f_{\eps} \le 1$ and $f_{\eps}$ is supported in $U_{\eps}$. Suppose that $\omega^{y}(K) = \eta > 0$. Then, letting $u_\eps$ be a weak solution of
$$
\begin{cases}
    L_{A} u_{\eps} = 0 & \text{ in } \Omega \\
    u_{\eps} = f_{\eps} & \text{ on } \partial \Omega
\end{cases}
$$
it follows from a Harnack chain argument that $u_{\eps}(y) \le C u_{\eps}(x)$ for some constant $C$ which depends on $x,y, \Omega,$ and the ellipticity of $A$. On the other hand, the integral representation of $u_\eps$ yields
$$
\begin{cases}
\eta \le u_{\eps}(y) = \int f_{\eps} d \omega^{y} \\
\eps \ge u_{\eps}(x) = \int f_{\eps} d \omega^{x}
\end{cases}
$$
so that $u_{\eps}(y) \ge \frac{\eta}{\eps} u_{\eps}(x)$ reaching a contradiction for $\eps \le C^{-1} \eta$.

Under more general assumptions than our own, the work of Grüter and Widman \cite{GW} ensures the existence of a Green function.\footnote{In fact they do not require the exterior corkscrew condition, and instead use the CDC condition (see e.g. \cite[Definition 2.6]{TZ}.} The book \cite{kenig1994harmonic} describes the behavior of the elliptic measure and corresponding Green functions. In particular the results proved in \cite{JK} for harmonic functions on NTA domains extend to solutions of $L_{A}$ on NTA domains (or even uniform domains with the CDC) with appropriate adaptations.  Throughout the remainder of the preliminaries, we work under the following assumption.

\begin{assumption}\label{a:main1}
    Let $\Lambda_0, R_{0} > 0$, $n \ge 3$, $M > 1$, and $\Omega \subset \R^{n}$ be an {$(M,R_{0})$- NTA domain.} Let $A\in L^\infty(\Omega;\R^{n\times n})$ be symmetric and $\Lambda_{0}$-elliptic matrix-valued function, and let $L_A = - \diverg(A\nabla \cdot)$. Let $\{\omega^x\}_{x\in \Omega}$ denote the family of elliptic measures associated to $L_A$ in $\Omega$. We refer to $\omega^x$ as the \emph{elliptic measure with pole at $x$}.
\end{assumption}

We summarize below the existence and basic properties of a Greens function. {The existence theory began in the bounded domain setting within \cite{GW} and was extended to systems and unbounded domains in \cite{HofmannKim}. Our presentation most closely matches  \cite[Lemma 3.4]{AzzMou}, which adds Theorem \ref{t:gfp}(4) and was first proven in variable coefficient setting in \cite{azzam2023uniform}.}

\begin{theorem}[{Green functions, \cite[Theorem 2.7]{TZ}}] \label{t:gfp}
Let $\Omega$, $A$, $L_{A}$, and $\{\omega^{x}\}$ be as in Assumption \ref{a:main1}. There exists a unique $G: \Omega \times \Omega \to [0,\infty]$ and a constant $C$ depending on the $(M,R_{0})$ such that the following hold:
\begin{enumerate}
    \item $0 \le G(x,y) \le C |x-y|^{2-n}$,
    \item $G(x,\cdot) \in C\left( \overline{\Omega}, \{x\} \right) \cap W^{1,2}_{\loc}\left(\Omega \setminus \{x\}\right)$ and $G(x, \cdot)|_{\partial \Omega} \equiv 0$,
    \item $G(x,y) = G(y,x)$, and
    \item for every $\varphi \in C_{c}^{\infty}(\R^{n+1})$
    $$
        \int_{\partial \Omega} \varphi d \omega^{x} = \varphi(x) - \int_{\Omega} \Langle A(y) \nabla_{y} G(x,y), \nabla \varphi(y) \Rangle d y
    $$
\end{enumerate}
\end{theorem}

We further recall from \cite{kenig1994harmonic} some basic results concerning the behavior of $L_A$-elliptic measures and functions on NTA domains, which are the analogue of properties derived in \cite{JK} for $L_A = -\Delta$. Our presentation closely follows \cite{TZ}. For $\Omega$ and $A$ as in Assumption \ref{a:main1}, we refer to allowable constants as constants depending only on the dimension $n$, the ellipticity $\Lambda_{0}$, the $L^{\infty}$-norm of $A$, and the NTA constants $(M,R_0)$ of $\Omega$.
Furthermore all of the following results still hold if the exterior corkscrew condition is replaced by the CDC condition.

\begin{lemma}\label{e:boundaryholder}
Let $\Omega$, $A$ and $L_A$ be as in Assumption \ref{a:main1}. There exist constants $\beta > 0$ and $C>0$ depending on allowable constants such that for any $p \in \partial \Omega$, $0< r < \diam \Omega$, the following holds. If $u \in W^{1,2}({\Omega};[0,\infty))$ with $L_A u = 0$ weakly in $B(p,2r) \cap \Omega$ and $u$ vanishes continuously\footnote{Since NTA domains are Wiener regular, it is well-known \cite{Wie,LSW} that solutions with continuous boundary data are continuous at the boundary. In particular, it is not restrictive to assume that $u$ vanishes continuously on $\Delta(p,2r)$.} on $\Delta(p,2r) = B(p,2r) \cap \partial \Omega$, then
$$
u(x) \le Cr^{-\beta}|x-p|^{\beta} \|u\|_{C^0(B(p,2r) \cap \Omega)} \qquad \forall x \in \Omega \cap B(p,r)\,.
$$
\end{lemma}

\begin{corollary}
    Let $\Omega$, $A$, $L_A$ and $\{\omega^x\}_{x\in \Omega}$ be as in Assumption \ref{a:main1}. There exists $m_{0} > 0$ depending on allowable constants such that for any $p \in \partial \Omega$ and $0< r < R_0$,
    $$
        \omega^{x_{p,r}}(B(p,r)\cap\partial\Omega) \ge m_{0}\,,
    $$
    where $x_{p,r}$ is an interior corkscrew point for $\Delta(p,r)$.
\end{corollary}

\begin{lemma}\label{l:boundaryharnack}
    Let $\Omega$, $A$ and $L_A$ be as in Assumption \ref{a:main1}. There exists $C>0$ depending on allowable constants such that for any $p \in \partial \Omega$ and $0< r < R_0$, the following holds. If $u \in W^{1,2}({\Omega};[0,\infty))$ with $L_A u = 0$ weakly in $B(p,4r)\cap\Omega$ and $u$ vanishes continuously on $\Delta(p,4r)$, then
    $$
        u(x) \le C u(x_{p,r}) \qquad \forall x \in \Omega \cap B(p,r)\,.
    $$
\end{lemma} 

\begin{lemma}\label{l:cfms}
Let $\Omega$, $A$, $L_A$ and $\{\omega^x\}_{x\in \Omega}$ be as in Assumption \ref{a:main1}. There exists $C > 0$ depending on allowable constants such that for all $p \in \partial \Omega$ and $0 < r < R_{0}/M$ if $L_{A} u = 0$ weakly in $B(p,4r) \cap \Omega$ and $u$ vanishes continuously on $\Delta(p,4r)$, then
    \begin{equation}
        C^{-1} \le \frac{\omega^{x}(\Delta(p,r))}{r^{n-2} u(x_{p,r})} \le C \qquad \text{for any } x \in \Omega \setminus B(p, 4r)\,.
    \end{equation}
    In particular, this holds when  $u(\cdot) = G(\cdot, x)$ is the Green function of $L_A$ in $\Omega$ with pole at $x$ as in Theorem \ref{t:gfp}.
\end{lemma}

\begin{lemma}\label{l:doubling}
    Let $\Omega$, $A$, $L_A$ and $\{\omega^x\}_{x\in \Omega}$ be as in Assumption \ref{a:main1}. There exists $C > 0$ depending on allowable constants such that for all $p \in \partial \Omega$ and $0 < r < R_0/2M$, the following holds. If $y \in \Omega \setminus B(p, 2Mr)$, then for $s \in (r/2,r)$,
    $$
        \omega^{y}(B(p,r)) \le C \omega^{y}(B(p,s))\,.
    $$
\end{lemma}


\subsection{\texorpdfstring{Tangent measures and $\Lambda$-tangents}{Tangent measures and Lambda-tangents}} \label{ss:tangents}
For a compact set $K \subset \R^{n}$, and two Radon measures $\mu,\nu$ we define
\begin{equation} \label{e:F}
F_{K}(\mu,\nu) = \sup \left\{ \left|\int f d(\mu-\nu) \right|\mid \Lip(f) \le 1,  ~ f \in C_{c}(K) \right\}.
\end{equation}
If $K = B_{r}$, we simply write $F_{r}( \cdot, \cdot)$. We recall, see \cite[Lemma 14.13]{Mattila} that for a sequence of Radon measures $\{\mu_{k}\}$ and a Radon measure $\mu$,
\begin{equation} \label{e:allradii}
\mu_{k} \xrightharpoonup{*} \mu \iff \lim_{k \to \infty} F_{r}(\mu_{k},\mu) = 0 \quad \forall r > 0.
\end{equation} 
It is well-known, see \cite[Proposition 1.12]{Preiss}, that
\begin{equation*}
F(\mu,\nu) := \sum_{\ell=1}^{\infty} 2^{-\ell} \min \{1, F_{\ell}(\mu,\nu)\}
\end{equation*}
defines a metric on the space of Radon measures. Moreover, $F$ generates the topology of weak-$*$ convergence. We denote $F(\mu) = F(\mu,0)$ and $F_r(\mu) = F_r(\mu,0)$.  

Throughout the remainder of this section we will suppose that $\Lambda : \R^{n} \to \R^{n \times n}$ is an invertible matrix-valued function. We denote this by writing $\Lambda : \R^{n} \to GL(n,\R)$.

 The notion of tangent measures that we will summarize stems from Preiss' resolution of the density question \cite{Preiss}. The extension of Preiss' machinery to the elliptic setting was introduced and first used in \cite{GTW-Lambda} by Casey, the first author, Toro, and Wilson. The presentation here will follow most similarly to \cite{GTW-Lambda}, but we emphasize that the case where $\Lambda \equiv \id$ always reduces to Preiss' work.

 Given an invertible matrix-valued function $\Lambda : \R^{n} \to GL(n,\R)$ fixed throughout this section, we consider the ellipse centered at $p$
 $$
        E_{\Lambda}(p,r) = p +  \Lambda(p)B(0,r),
 $$
 whose eccentricity depends on $p$. We further consider the rescalings
 $$
T^{\Lambda}_{p,r}(y) = \Lambda(p)^{-1} \left( \frac{y-p}{r} \right),
$$
and the associated pushforward measures $T_{p,r}^{\Lambda}[\mu]$. The latter in particular satisfy
$$
T^{\Lambda}_{p,r}[\mu](B_{1}) = T_{p,r}[\mu](E_{\Lambda}(0,1)) = \mu\left(E_{\Lambda}(p,r)\right).
$$

\begin{definition}[$\Lambda$-tangents]
If $\mu$ is a Radon measure on $\R^n$, $p\in \spt\mu$ and $\Lambda:\R^n\to GL(n,\R)$, we define the \emph{class of $\Lambda$-tangents to $\mu$ at $p$} as
\begin{equation} \label{e:lambdatan}
\Tan_{\Lambda}(\mu,p) := \left\{\nu  \text{ Radon s.t. } \nu = \lim_{i} c_{i} T^{\Lambda}_{p,r_{i}}[\mu]  : c_{i} > 0, \, r_{i} \downarrow 0, \, \nu \neq 0 \right\}.
\end{equation}

In the special case where $\Lambda = \id$, the class of $\Lambda$-tangents at a given point is simply the class of tangent measures and thus we simply write $\Tan(\mu,p)$; see the discussion in the introduction. Furthermore, we denote by $\Tan_{\Lambda}[\mu]$ the weak-$*$ closure of the space of measures $\cup_{p \in \spt \mu} \Tan_{\Lambda}(\mu,p)$. We essentially only consider this space in the case of $\Tan_{\id}[\mu]$ which we will simply denote as $\Tan[\mu]$.
\end{definition}

\begin{definition}[Cones, $d$-cones, and basis] \label{d:dcone}
A collection of nonzero Radon measures $\cM$ is called a \emph{cone} if 
$$
    \mu \in \cM \implies c \mu \in \cM \qquad \text{for all $c > 0$\,.}
$$
A cone of Radon measures is called a  $d$-\emph{cone} or dilation cone if
$$
    \mu \in \cM \implies T_{0,r}[\mu] \in \cM \qquad \text{for all $r > 0$\,.}
$$
The \emph{basis} of a $d$-cone is the collection of $\mu \in \cM$ so that $F_{1}(\mu) = 1$. 
A $d$-cone $\cM$ is said to have a \emph{closed} (respectively \emph{compact}) \emph{basis} if the basis is closed (respectively compact) with respect to the weak-$*$ topology.
\end{definition}

We next introduce a notion of the distance to a $d$-cone. Let $\cM$ be a $d$-cone and $\nu$ a Radon measure in $\R^{n}$. If $s> 0$ and $0< F_{s}(\nu) < \infty$ we define the \emph{distance} between $\nu$ and $\cM$ at scale $s$ by
\begin{equation} \label{e:ds}
d_{s}(\nu,\cM) = \inf \left\{ F_{s} \left( \frac{\nu}{F_{s}(\nu)}, \mu \right) : \mu \in \cM ~ \text{and} ~ F_{s}(\mu) = 1 \right\}\,.
\end{equation}
If $F_{s}(\nu) \in \{0, \infty\}$, we define $d_{s}(\nu,\cM) = 1$.

\begin{proposition} \cite[Remarks 2.1 \& 2.2]{KPT} \label{p:kptrs}
If $\mu, \nu$ are Radon measures, 
\begin{equation} \label{e:fscaling}
F_{r}(\mu,\nu) = r F_{1}(T_{0,r} [\mu], T_{0,r}[\nu]).
\end{equation}

If $\cM$ is a $d$-cone and $\nu$ a Radon measure,
\begin{enumerate}
\item[i)] $d_{s}(\nu, \cM) \le 1$ for all $s > 0$.
\item[ii)] $d_{s}(\nu, \cM) = d_{1} \left( T_{0,s}[\nu],\cM\right)$ for all $s > 0$.
\item[iii)] If $\nu_{i} \xrightharpoonup{*} \nu$ and $F_{s}(\nu) > 0$, then $d_{s}(\nu,\cM) = \lim_{i \to \infty} d_{s}(\nu_{i},\cM)$.
\end{enumerate}
\end{proposition}

\begin{remark}
Given $\nu \in \Tan_{\Lambda}(\mu,p)$ with $c_{i} T^{\Lambda}_{p,r_{i}}[\mu] \xrightharpoonup{*} \nu$ it is easy to check that $c T_{0,r}[\nu] = \lim_{i} c c_{i} T^{\Lambda}_{p, rr_{i}}[\mu]$ for any $c , r > 0$. In particular $\Tan_{\Lambda}(\mu,p)$ is a $d$-cone.
\end{remark}

A crucial property of $\Lambda$-tangents (and in particular tangent measures) is that tangents to $\Lambda$-tangents are $\Lambda$-tangents. This is a consequence of the following theorem.

\begin{theorem}{\cite[Theorem 3.3]{GTW-Lambda}} \label{t:tan2ltan}
Let $\mu$ be a Radon measure on $\R^{n}$ and $\Lambda : \R^{n} \to GL(n,\R)$. Then at $\mu$-a.e. $p \in \R^{n}$, every $\nu \in \Tan_{\Lambda}(
\mu,p)$ has the following two properties:
\begin{enumerate}
\item $T_{x,r}[\nu] \in \Tan_{\Lambda}(\mu,p)$ for all $x \in \spt \nu, r > 0$.
\item $\Tan(\nu,x) \subset \Tan_{\Lambda}(\mu,p)$ for all $x \in \spt \nu$.
\end{enumerate}
\end{theorem}

$\Lambda$-tangents were introduced as a tool to transform anisotropic information of a base measure into an isotropic form for the $\Lambda$-tangents. It turns out this is equivalent to taking a tangent measure, and then performing a linear rigid change of variables. In this paper, we use the latter approach in order to allow access to all the estimates from Section \ref{s:emb} while performing a blow-up without restating Section \ref{s:emb} in terms of ellipses (which would also be possible). To formalize this geometric intuition, we recall the Isomorphism Lemma:

\begin{lemma}{\cite[Lemma 3.4]{GTW-Lambda}} \label{l:iso}
Let $\mu$ be a Radon measure on $\R^{n}$ and $\Lambda : \R^{n} \to GL(n,\R)$. For a Radon measure $\nu$, the following are equivalent:
\begin{enumerate}
\item $\nu \in \Tan_{\Lambda}(\mu,p)$
\item $\Lambda(p)_{\sharp} \nu \in \Tan(\mu,p)$
\item $\nu \in \Tan( (\Lambda(p)^{-1})_{\sharp} \mu,  \Lambda(p)^{-1} p)$
\end{enumerate}
\end{lemma}

\begin{remark}
    In Theorem \ref{t:decomp} (i), one of our goals is to conclude that $\Tan(\mu,p) \subset \cF$, the latter being the space of flat measures. Since for any invertible matrix $\Lambda(p)$, $\Lambda(p)_{\sharp} \sigma \in \cF$ for all $\sigma \in \cF$, the equivalence of (1) and (2) in Lemma \ref{l:iso} immediately shows that this conclusion in Theorem \ref{t:decomp} can equivalently be stated for $\Tan_{\Lambda}(\mu,p)$ in place of $\Tan(\mu,p)$.
\end{remark}

The next lemma states that $\Tan_{\Lambda}(\mu,p)$ is connected in a specific sense:

\begin{lemma}{\cite[Corollary 3.6]{GTW-Lambda}} \label{l:connectedness}
    Let $\Lambda : \R^{n} \to GL(n,\R)$. Suppose $\cF = \cup_{i=1}^{\infty} \cF_{i}$, each $\cF_{i}$ is a $d$-cone with compact basis and there exists a $d$-cone $\cM$ with closed basis so that $\cF \subset \cM$ and for each $i$ the following holds
    \begin{equation} \label{e:pi} \tag{$P{_i}$}
        \begin{cases}
            \exists\, \eps_{i} > 0, \, R_{i} > 0 \text{ such that } \forall \, \eps \in (0, \eps_{i}) \text{ there exists no } \nu \in \cM \setminus \cup_{j=1}^{i-1} \cF_{j} \\
            \text{ satisfying } d_{r}(\nu, \cF_{i}) \le \eps ~ \forall \, r \ge R_i > 0 ~ \text{ and } d_{R_i}(\nu, \cF_{i}) = \eps. 
        \end{cases}
    \end{equation}
    Then for any $p \in \R^{n}$ so that $\Tan_{\Lambda}(\mu,a) \subset \cM$, the following dichotomy holds:
    \begin{enumerate}
        \item[(i)] $\Tan_{\Lambda}(\mu,p) \subset \cF$, or 
        \item[(ii)] $\Tan_{\Lambda}(\mu,p) \cap \cF = \emptyset$. 
    \end{enumerate}
\end{lemma}

We recall Definition \ref{d:cda1a2}, where we defined $\cD(A^{+},A^{-})$ to be the class of Radon measures with the following properties:
\begin{enumerate}
    \item[a)] $\R^{n} \setminus \spt ~ \nu = \Omega^{+} \cup \Omega^{-}$ for unbounded complementary NTA domains $\Omega^{\pm}$ with fixed NTA constants.
    \item[b)] There exist non-negative $u^{\pm} > 0$ that are Green's functions with pole at infinity for the triple $(\Omega^{\pm} , \nu, L_{A^{\pm}})$.
\end{enumerate}

\section{Reduction to tangents} \label{s:blowup}

In this section we will adapt a blow-up argument, pioneered by Kenig and Toro \cite{KT03-Poisson,kenig2006free} to our setting. Throughout, we will work under the following assumption.

\begin{assumption}\label{a:main2}
    Let $n \ge 3$ and $\Lambda_0 > 0$. Let $\Omega^{+}$ and $\Omega^- = \R^n\setminus \overline{\Omega^+}$ be complementary $(M,R_{0})$-NTA domains with $\partial\Omega := \partial \Omega^+ = \partial \Omega^-$, let $A^{\pm} \in L^{\infty}(\R^{n} ; \R^{n \times n})$ and let $A^{\pm}$ be $2$-quasicontinuous, symmetric, and $\Lambda_{0}$-elliptic matrix-valued functions, and let $L_{A^\pm} = - \diverg(A^\pm\nabla \cdot)$ in $\Omega^\pm$. Let $x_0^\pm \in \Omega^\pm$ be fixed poles, and let $u^\pm := G^\pm(\cdot, x_0^\pm)$ and $\omega^\pm := \omega^{\pm, x_0^\pm}$ be the respective Green functions and elliptic measures for $L_{A^\pm}$ in $\Omega^\pm$ with poles at $x_0^\pm$.
\end{assumption}

Let $\Lambda_{0}, \Omega^{\pm}, u^{\pm}, A^{\pm}$, and $\Omega^{\pm}$ be as in Assumption \ref{a:main2}. Given a sequence of positive numbers $\{r_i\}$ with $r_i \downarrow 0$ and $p \in \partial \Omega$, define
\begin{equation} \label{e:domainseq}
    \Omega_{i}^{\pm} = \frac{\Omega^\pm - p}{r_{i}}\,, \qquad \partial\Omega_i = \frac{\partial\Omega - p}{r_{i}} 
    \end{equation}
together with the functions
\begin{equation} \label{e:functionseq}
    u^{\pm}_{i}(x) = \frac{u^{\pm}( r_{i} x + p)}{\omega^{\pm}(B(p,r_{i}))} r_{i}^{n-2}
\end{equation}
and  the measures
\begin{equation} \label{e:measureseq}
    \omega^{\pm}_{i}(E) = \frac{\omega^{\pm}(r_{i} E + p)}{\omega^{\pm}(B(p,r_{i}))} \quad \text{for Borel subsets $E\subset \R^{n}$}.
\end{equation}
Note that \eqref{e:measureseq} may be rewritten as $\omega^{\pm}_{i} = c_{i}T_{p,r_{i}}[\omega^{\pm}]$ where $c_{i} = \omega^{\pm}(B(p,r_{i}))^{-1}$.

Before beginning the proof in earnest, we take note of a few properties of the defined sequences. By Lemma \ref{l:cfms}, for $i$ large enough (so that $r_{i} \le R_{0}/C_{0}$ and our unnamed pole is outside $B(p,4r_{i})$) it follows that there exists a $C >0$ depending on the NTA constants, dimension, and ellipticity (but not on $i$) so that
\begin{equation} \label{e:fsb1}
C^{-1} \le \frac{u^{\pm}(x^{\pm}_{p,r_{i}})}{\omega^{\pm}(B(p,r_{i}))} r_{i}^{n-2} \le C\,
\end{equation}
where $x^{\pm}_{p,r_{i}} \in \Omega^\pm$ are interior corkscrew points for $\Omega^\pm$ associated to the surface balls $\Delta(p,r_i)$. 
In addition, the boundary harnack property for NTA domains, Lemma \ref{l:boundaryharnack} yields that for fixed $N > 1$ and all $x \in B(0,N)$, whenever $i$ is large enough depending on $N$,
\begin{equation} \label{e:fsb2}
u^{\pm}(r_{i} x + p) \le C u^{\pm}(x^{\pm}_{p,r_{i}}).
\end{equation}

In particular, combining \eqref{e:functionseq}, \eqref{e:fsb1}, and \eqref{e:fsb2} we conclude that for all $x \in B(0,N)$:
\begin{equation} \label{e:fseqlb}
\sup_{i \ge 1} \sup_{x \in B(0,N)} u^{\pm}_{i}(x)\le C < \infty.
\end{equation}

Furthermore, since $\omega^{\pm}$ are locally doubling (Lemma \ref{l:doubling}), it follows from \eqref{e:measureseq} that
\begin{equation} \label{e:mseqlb}
    \sup_{i \ge 1} ~ \omega^{\pm}_{i}(B(0,N)) = \sup_{i \ge 1} \frac{\omega^{\pm}(B(p,r_{i}N))}{\omega^{\pm}(B(p,r_{i}))} \le C_{N} <  \infty
\end{equation}
where $C_{N}>0$ depends on $N$ and allowable constants. 

Finally, to state the main theorem of this section, we decompose the boundary $\partial\Omega$ of $\partial \Omega^{\pm}$. 

\begin{equation} \label{e:gamma1}
    F_1 = \left \{ p \in \partial \Omega : 0 < h(p) := \frac{d \omega^{-}}{d \omega^{+}}(p) = \lim_{r \to 0} \frac{\omega^{-}(B(p,r))}{\omega^{+}(B(p,r))} < \infty \right\},
\end{equation}

\begin{equation} \label{e:gamma2}
    F_2 = \left\{ p \in \partial\Omega: \frac{d \omega^{-}}{d \omega^{+}}(p) = \infty \right\} 
\end{equation}
\begin{equation} \label{e:gamma3}
    F_3 = \left \{ p \in \partial\Omega : \frac{d \omega^{-}}{d \omega^{+}}(p) = 0 \right\}
\end{equation}
\begin{equation} \label{e:gamma4}
    F_4 = \left \{ p \in \partial\Omega : \frac{d \omega^{-}}{d \omega^{+}}(p) \text{ does not exist } \right\}.
\end{equation}

Given a Radon measure $\mu$ and a $\mu$-measurable set $E$, we denote by $\Theta(\mu, p, E)$ the density
\[
    \theta(\mu, p, E):= \lim_{r\downarrow 0} \frac{\mu(B(p,r)\cap E)}{\mu(B(p,r))}\,,
\]
whenever such a limit exists. In addition, given a $\mu$-measurable function $f$, we say that $x$ is a point of $\mu$-approximate continuity for $f$ if
$$
f(x) = \lim_{r \to 0} \aveint{B(x,r)}{} f d \mu \quad \text{and} \quad \lim_{r \to 0} \aveint{B(x,r)}{} |f(x) - f(y)| d\mu(y) = 0.
$$
We denote the set of $\mu$-approximate continuity points for $f$ by $\Ccal(\mu,f)$, and define 
\begin{align} \label{e:gamma0}
    F_0 & = \bigg\{ p \in F_1 \cap \Ccal\Big(\omega^+,\frac{d \omega^{-}}{d \omega^{+}}\Big) \cap \Ccal(\omega^\pm, A^\pm) : \theta(\omega^\pm,p,F_1) \text{ exists and equals 1} \bigg\}\,.
\end{align}

\begin{remark}\label{r:decomp-Leb-diff}
By the Lebesgue-Besicovitch-Federer differentiation theorem, we know that:
\begin{itemize}
    \item $\partial\Omega = F_1 \sqcup F_2 \sqcup F_3 \sqcup F_4$;
    \item $\omega^{+}(F_2)= 0 = \omega^{-}(F_3) = \omega^{\pm}(F_4)$;
    \item $\omega^{+} \mres F_1$ and $\omega^{-} \mres F_1$ are mutually absolutely continuous;
    \item $\frac{d \omega^{-}}{d \omega^{+}} \in L^{1}_{\loc}(\omega^{+})$ and $\frac{d \omega^{+}}{d \omega^{-}} = 1/ \left( \frac{d \omega^{-}}{d \omega^{+}} \right) \in L^{1}_{\loc}(\omega^{-})$;
    \item $\omega^{\pm}(F_1 \setminus F_0) = 0$.
\end{itemize}
Note that the last property additionally uses the fact that $A^{\pm}$ are $2$-quasicontinuous and \cite[Theorem 11.14]{HKT}.
\end{remark}

We introduce a final piece of notation which is subtle, but crucial to succinctly talk about the results of this section. We recall that we often have two matrix-valued functions $A^{\pm}$, and the two corresponding operators 
$$
L_{A^{\pm}} u = -\divr (A^{\pm}(\cdot) \nabla u(\cdot)) = - \divr(A^{\pm} \nabla u)
$$
where $(\cdot)$ should be filled with the spatial variable. In contrast, letting
$$
    A(u,x) := 
        \begin{cases}
            A^{+}(x) & u > 0 \\ 
            A^{-}(x) & u \le 0,
        \end{cases}
$$
we define operator
\begin{equation}\label{e:L^A}
L^{A}u := - \divr (A(u, \cdot) \nabla u( \cdot)) = - \divr(A(u) \nabla u)\,.
\end{equation}

We are now ready to state the main result concerning blow-ups of our two-phase problem. {This will allow us to reduce to the case of a constant coefficient problem on unbounded NTA domains.}

\begin{theorem}\label{t:blowups}
Let $\Omega^{\pm}$, $\partial\Omega$, $A^{\pm}$, $L_{A^\pm}$, $u^\pm$ and $\omega^\pm$ be as in Assumption \ref{a:main2}. Using the notation above, for any $p \in F_0$, there exists a subsequence (which we do not relabel) so that as $i \to \infty$ the following converge in Hausdorff distance uniformly on compact sets:
\begin{equation} \label{e:domainconv}
    \Omega_{i}^{\pm} \rightarrow \Omega^{\pm}_{\infty} \quad \text{and} \quad \partial\Omega_{i} \rightarrow \partial\Omega_{\infty}\,,
\end{equation}
where $\Omega^{\pm}_{\infty}$ are unbounded NTA domains with the same NTA constants as $\Omega^{\pm}$ and satisfy $\partial \Omega^{+}_{\infty} = \partial \Omega^{-}_{\infty} = \partial\Omega_\infty$. Moreover, there exist $u^{\pm}_{\infty} \in C(\R^{n})$ (extended by zero) such that
\begin{equation} \label{e:funcconv}
u^{\pm}_{i} \rightarrow u^{\pm}_{\infty} \quad \text{{in $W^{1,2}_{\loc}(\R^{n})$ and} uniformly on compact sets}
\end{equation}
and $u := u^+_\infty - u^-_\infty$ is a weak $W^{1,2}_{\loc}(\R^{n})$-solution to 
\begin{equation} \label{e:globalpde}
L^{A_{p}} u = 0 \quad \text{ on } \, \R^{n}\,,
\end{equation}
where $A_p(u) := A^+(p) \mathbf{1}_{\{u>0\}} + A^-(p)\mathbf{1}_{\{u\leq 0\}}$\footnote{Note that $A^\pm(p)$ are well-defined since $p\in F_0$.} and $L^{A_p}$ is as in \eqref{e:L^A}. 

Furthermore, there exist Radon measures $\omega^{\pm}_{\infty}$ so that
\begin{equation} \label{e:measureconv}
    \omega^{\pm}_{i} \toweakstar \omega^{\pm}_{\infty}\,,
\end{equation}
and $\omega^{\pm}_{\infty}$ are the elliptic measures of $\Omega^{\pm}_{\infty}$ (with pole at infinity) corresponding to the operators $L_{A^{\pm}(p)}$. Moreover, $\omega^{+}_{\infty} = \omega^{-}_{\infty} =: \omega_{\infty}$. That is, for all $\varphi \in C^{\infty}_{c}(\R^{n})$,
\begin{equation} \label{e:gfident}
\int_{\Omega^{\pm}_{\infty}} A^\pm(p) \nabla u^{\pm}_{\infty} \cdot \nabla \varphi = \int_{\partial\Omega_\infty} \varphi d \omega_{\infty}\,,
\end{equation}
{and
    \begin{equation} \label{e:onesided}
    \begin{cases}
        L_{A^{\pm}(p)} u^{\pm}_{\infty} = 0 & \text{ on } \Omega^{\pm}_{\infty} \\
        u^{\pm}_{\infty} = 0 & \text{ on } \R^{n} \setminus \Omega^{\pm}_{\infty} \\
        u^{\pm}_{\infty} > 0 & \text{ on }  \Omega^{\pm}_{\infty}\,. 
    \end{cases}
    \end{equation}}
{Finally, we note that there exists $C>0$ depending on the NTA constants so that
\begin{equation} \label{e:normalization}
    0 \in \spt \omega_{\infty} \quad \text{and} \quad C^{-1} \le  \omega_{\infty}(B_{1}) \le 1.
\end{equation}
}
\end{theorem}

Notice that Theorem \ref{t:blowups} is a more detailed re-statement of Theorem \ref{t:maintangents}; in other words, we are concluding that $\wslim \omega_{i}^{\pm} = \omega^{\infty} \in \cD(A^+(p), A^{-}(p))$ and \eqref{e:gfident} holds.

{
\begin{remark} \label{r:poleatinf}
For each choice of sign $\pm$, we refer to the members of the triple $(\Omega^{\pm}, u_{\infty}^{\pm},  \omega_{\infty})$ in the conclusion of Theorem \ref{t:blowups} as follows. The functions $u_{\infty}^{\pm}$ are the Green's functions on $\Omega^{\pm}_{\infty}$ with poles at infinity, and $\omega_{\infty}$ is the corresponding harmonic measure with pole at infinity.

It is a subtle distinction that we only mention the Green's function/harmonic measure with pole at infinity in our blown-up setting; this is because Theorem \ref{t:blowups} proves the existence of such objects after a blow-up. For the case of the Laplacian, it is known that these objects exist for unbounded NTA domains \cite[Corollary 3.2]{kenig1999free}, but while we expect it to be true, we could not find and do not require an analogous existence theorem in the variable coefficient setting.
\end{remark}
}
\begin{proof}
    Thanks to pre-compactness following from \eqref{e:fseqlb} and \eqref{e:mseqlb}, the fact that there exists a subsequence and limiting objects so that \eqref{e:domainconv}, \eqref{e:measureconv}, {and the ``uniformly on compact sets" part of} \eqref{e:funcconv}, all hold follows as in \cite[Section 4]{KT03-Poisson}. {The $W^{1,2}_{\loc}$ part of \eqref{e:funcconv} follows as in \cite[Lemma 3.11]{AzzMou}}. The fact that $0 \in \spt \omega_{\infty}$ is a generic fact about tangent measures. That $\omega_{\infty}(B_{1}) \le 1$ follows from \eqref{e:measureseq} and the characterization of weak-$*$ convergence in terms of open sets. The fact that $C^{-1} \le \omega_{\infty}(B_{1})$ follows from the fact that the sequence $\omega_{i}$ is uniformly doubling (Lemma \ref{l:doubling}) with $\omega_{i}(B_{1}) = 1$ implying $\omega_{i}(\overline{B_{3/4}}) \ge C^{-1}$ for all $i$, and then using the characterization of weak-$*$ convergence in terms of compact sets.
    
    Moreover, we claim that (cf. \cite[Lemma 4.11 (f) \& (4-11)]{AzzMou})
    \begin{equation} \label{e:bu2}
        \int_{\Omega^{\pm}_{\infty}} A^\pm(p) \nabla u^{\pm}_{\infty} \cdot \nabla \varphi = \int_{\partial\Omega_\infty} \varphi d \omega^{\pm}_{\infty} \qquad \forall \ \vphi\in C_c^\infty(\R^n)\,,
    \end{equation}
    {and \eqref{e:onesided} holds.} {More specifically, by definition of a point of $\omega^\pm$-approximate continuity of $A^\pm$ at $p$, we have
    \begin{equation}\label{e:appr-cty-omega}
         \lim_{r \to 0} \frac{1}{\omega^\pm(B(p,r_i))}\int_{B(p,r)} |A^\pm(p) - A^\pm(y)| d\omega^\pm(y) = 0\,.
    \end{equation}
    Combining this with the local weak-$*$ convergence of $\omega_i^\pm$ to $\omega^\pm_\infty$, for a given $\vphi\in C_c^\infty(\R^n)$, letting $\vphi_i:= \vphi(\tfrac{\cdot - p}{r_i})$ for $i$ large enough we have
    \begin{align*}
       \int_{\partial\Omega_\infty} \varphi d \omega^{\pm}_{\infty} &= \int \vphi d\omega^\pm_i \\
       &= \frac{1}{\omega^\pm(B(p,r_i))} \int_{\partial\Omega^\pm} \vphi_i d\omega^\pm \\
       &= \frac{r_i^n}{\omega^\pm(B(p,r_i))} \int_{\Omega_i^\pm} A^\pm(p + r_i y) \nabla u^\pm(p + r_i y) \cdot \nabla \vphi (y) \\
       &= \frac{r_i^n}{\omega^\pm(B(p,r_i))} \int_{\Omega_i^\pm} A^\pm(p) \nabla u^\pm(p + r_i y) \cdot \nabla \vphi (y) \\
       &\qquad+ \frac{r_i^n}{\omega^\pm(B(p,r_i))} \int_{\Omega_i^\pm} (A^\pm(p + r_i y) - A^\pm(p))\nabla u^\pm(p + r_i y) \cdot \nabla \vphi (y) \\
        &= \int_{\Omega_\infty^\pm} A^\pm(p) \nabla u_\infty^\pm \cdot \nabla \vphi + \int_{\Omega_i^\pm \Delta \Omega_\infty^\pm} A^\pm(p) \nabla u_\infty^\pm \cdot \nabla \vphi \\
       &\qquad+ \int_{\Omega_i^\pm} A^\pm(p) (\nabla u_i^\pm - \nabla u_\infty^\pm ) \cdot \nabla \vphi \\
       &\qquad+ \int_{\Omega_i^\pm} (A^\pm(p + r_i y) - A^\pm(p))\nabla u_i^\pm(y) \cdot \nabla \vphi (y)\,.
    \end{align*}
    A combination of \eqref{e:domainconv}, \eqref{e:funcconv} and \eqref{e:appr-cty-omega} ensures that all but the first term in the final equality converge to zero (up to extracting a subsequence), as desired. Here, we have used the definition of the Greens functions together with the fact that, since $\vphi$ has fixed compact support, for $i$ large enough the implicit fixed poles $x_0^\pm$ lie outside of the support of $\omega^\pm_i\mres \supp(\vphi)$.
    }
    
    
    While similar statements are known in the literature, see in particular \cite{KPT,KT03-Poisson, kenig2006free, AzzMou}, we prove that under the assumption that $p \in F_0$ (as defined in \eqref{e:gamma0}) both \eqref{e:globalpde} and \eqref{e:gfident} hold for our pair of operators with a discontinuity in the coefficients across the boundary. In light of \eqref{e:bu2} and \eqref{e:onesided}, it suffices to prove that $\omega^{+}_{\infty} = \omega^{-}_{\infty}$.

    To confirm $\omega^{+}_{\infty} = \omega^{-}_{\infty}$ we first observe that by duality, our rescaled measures \eqref{e:measureseq} satisfy
    \begin{equation} \label{e:bu1}
        \int_{\partial\Omega_{i}} \varphi ~ d \omega^{\pm}_{i} = \frac{1}{\omega^{\pm}(B(p,r_{i}))} \int_{\partial\Omega} \varphi \left( \frac{q-p}{r_{i}} \right) d \omega^{\pm}(q) \qquad \forall \vphi \in C_c^\infty(\R^n)\,.
    \end{equation}
    In particular, recalling the definition of $F_1$ from \eqref{e:gamma1},
    \begin{align*}
        \int_{\partial\Omega_i} \varphi \, d \omega^{-}_{i} & = \frac{1}{\omega^{-}(B(p,r_{i}))} \int_{F_1} \varphi \left( \frac{q-p}{r_{i}} \right) \frac{d \omega^{-}}{d \omega^{+}}(q) d \omega^{+}(q) \\
        \nonumber        & + \frac{1}{\omega^{-}(B(p,r_{i}))} \int_{\partial\Omega \setminus F_1} \varphi \left( \frac{q-p}{r_{i}} \right) d \omega^{-}(q) \\
    &=: (I) + (II).
    \end{align*}
    We note that as $i \to \infty$ the size of $(II)$ goes to zero. Indeed, combine the doubling of $\omega^{-}$ \eqref{e:mseqlb} with the compact support of $\varphi$ and the fact that $p \in F_0$ requires $\Theta(\omega^-,p,\partial \Omega \setminus F_1) = 0$. 
    Denoting $h = \frac{d \omega^{-}}{d \omega^{+}}$, we next write
    \begin{align*}
        (I) & = \frac{\omega^{+}(B(p,r_{i}))}{\omega^{-}(B(p,r_{i}))} \frac{1}{\omega^{+}(B(p,r_{i}))} \int_{F_1} \varphi \left( \frac{q-p}{r_{i}} \right) h(q) d \omega^{+}(q) \\
        & = h(p) \frac{\omega^{+}(B(p,r_{i}))}{\omega^{-}(B(p,r_{i}))} \frac{1}{\omega^{+}(B(p,r_{i}))} \int_{F_1} \varphi \left( \frac{q-p}{r_{i}} \right)  d \omega^{+}(q) \\
        & + \frac{\omega^{+}(B(p,r_{i}))}{\omega^{-}(B(p,r_{i}))} \frac{1}{\omega^{+}(B(p,r_{i}))} \int_{F_1} \varphi \left( \frac{q-p}{r_{i}} \right) (h(q)-h(p)) d \omega^{+}(q) \\
        & =: (III) + (IV).
    \end{align*}
     Since $p$ is a point of approximate continuity for $h$, \eqref{e:mseqlb} ensures that $(IV)$ goes to zero as $i \to \infty$. Finally, we see that
     $$
        \lim_{r \to 0} \frac{\omega^{+}(B(p,r_{i}))}{\omega^{-}(B(p,r_{i}))} = \frac{1}{h(p)}
     $$
     so that, when combined with \eqref{e:measureconv} and \eqref{e:bu1}, $(III)$ converges to 
    $\int_{\partial \Omega^{+}_{\infty}} \varphi d \omega^{+}_{\infty}$. In summary, we have shown
    $$
        \int_{\partial \Omega^{-}_{\infty}} \varphi d \omega^{-}_{\infty} = \lim_{i \to \infty} \, (II) + (III) + (IV) = \lim_{i \to \infty} \, (III) = \int_{\partial \Omega^{+}_{\infty}} \varphi d \omega^{+}_{\infty} \qquad \forall \varphi \in C_{c}^{\infty}(\R^{n})
    $$
    verifying that $\omega^{+}_{\infty} = \omega^{-}_{\infty} = \omega_{\infty}$ and consequently \eqref{e:gfident}. 
    
    We now turn to \eqref{e:globalpde}. We deduce from Theorem \ref{t:gfp} that $u \in W^{1,2}_{\loc}(\R^{n})$ since it is a difference of two Green functions with pole at infinity. Thus, from \eqref{e:onesided}, for $\varphi \in C_{c}^{\infty}(\R^{n})$:
    \begin{align*}
        \int_{\R^{n}}  \langle A_p(u) \nabla u, \nabla \varphi \rangle & = \int_{\Omega^{+}_\infty} \langle A^+(p)\nabla u, \nabla \varphi \rangle - \int_{\Omega^{-}_\infty} \langle {A}^{-}(p) \nabla u, \nabla \varphi \rangle \\
        & = \int_{\partial\Omega_\infty} \varphi d \omega^{+}_\infty - \int_{\partial\Omega_\infty} \varphi d \omega^{-}_\infty = 0
    \end{align*}
    verifying \eqref{e:globalpde}.
\end{proof}

{We record here a more general form of the first part of Theorem \ref{t:blowups}, which gives compactness for general varying sequences of domains with associated elliptic measures for a fixed operator, rather than those coming from rescalings of a single domain. This will come in useful later, when we implement Preiss' techniques in order to conclude the proof of Theorem \ref{t:maintangents} (see Proposition \ref{p:compactbasis}). Therein, we will additionally justify compactness for the Greens functions. We omit the proof, since it follows by the very same reasoning as that in the beginning of Theorem \ref{t:blowups}. More details can be found in \cite[Proposition 1.7]{MMP-layer-ptls}, which in the even more general setting of a CDC domain and varying operators.

\begin{theorem}\label{t:diagonal-cptness}
	Let $\{\Omega^{\pm}_i\}_i$ be a sequence of complementary $(M,R_0)$-NTA domains with boundaries $\partial\Omega_i$, let $A^{\pm}$ and $L_{A^\pm}$ be as in Assumption \ref{a:main2}, and let $u^\pm_i$ and $\omega^\pm$ also be as in Assumption \ref{a:main2} for $\Omega^\pm_i$. There exists a subsequence (which we do not relabel) so that as $i \to \infty$,
	\begin{equation*}
		\Omega_{i}^{\pm} \rightarrow \Omega^{\pm}_{\infty} \quad \text{and} \quad \partial\Omega_{i} \rightarrow \partial\Omega_{\infty}\,,
	\end{equation*}
	where the convergence is locally uniformly in Hausdorff distance, and $\Omega^{\pm}_{\infty}$ are unbounded $(M,R_0)$-NTA domains and satisfy $\partial \Omega^{+}_{\infty} = \partial \Omega^{-}_{\infty} = \partial\Omega_\infty$. Moreover, there exists $C>0$ depending on the NTA constants so that
		\begin{equation*}
			0 \in \spt \omega_{\infty} \quad \text{and} \quad C^{-1} \le  \omega_{\infty}(B_{1}) \le 1.
		\end{equation*}
\end{theorem}
}

\section{Free boundary interpretation of the blow-up domains} \label{s:freeboundary}
We are now in the position to study the free boundary problem \eqref{e:globalpde}.
We begin by using (\cite[Lemma 3.8, 3.9]{AzzMou}), which we state in a simplified version that is sufficient for our needs:

\begin{lemma}[{\cite[Lemma 3.8, Lemma 3.9]{AzzMou}}] \label{l:amcov} 
    Let $\Omega$, $A$, $L_A$ and $\{\omega^x\}_{x\in \Omega}$ be as in Assumption \ref{a:main1}. Let $p\in \partial\Omega$ and let $\Lambda(p)\in \R^{n\times n}$ be a $\Lambda_0$-elliptic matrix.
     Let $u= G(\cdot, x)$ be a Green function for $L_A$ on $\Omega$ with pole at $x\in\Omega$, and let $T(x) = \Lambda(p)^{-1}(x-p)$. Then
    $u \circ T^{-1}, T[\omega^x]$, are respectively a Green function and elliptic measure with pole at $T(x)$ for $L_{\tilde{A}}$ on $T(\Omega)$, where $\tilde{A}(y) = \Lambda(p)^{-1} A(\Lambda(p)y + p) \Lambda(p)^{-1} $.
\end{lemma}

We note that in the more general setting of \cite{AzzMou}, the above lemma contains a factor of $|\det(\Lambda(p))|$ in the definition of $\tilde{A}$. However, since this is a constant in our setting, it can be removed for aesthetic purposes. As an immediate corollary, we obtain the following.
\begin{corollary}\label{c:Lambda-tangent-FB}
    Suppose that $\mu$ is a Radon measure on $\R^n$, and let $\Omega$, $A$ be as in Assumption \ref{a:main1}. Then for every $p\in \partial\Omega$, the following holds. If $\Tan(\mu,p) \subset \cD(A^{+}(p), A^{-}(p))$, $\Lambda(p) = \sqrt{A^{+}(p)}$ and $M_p = \Lambda(p)^{-1} A^{-}(p) \Lambda(p)^{-1}$ then $\Tan_{\Lambda}(\mu,p) \subset \cD(\id, M_{p})$.
\end{corollary}

The corollary follows directly from the equivalence of (1) and (2) in Lemma \ref{l:iso} and Lemma \ref{l:amcov}. Therefore, by working with $\Tan_{\Lambda}(\omega^{\pm},p)$ instead of $\Tan(\omega^{\pm},p)$, we reduce to working under the following assumption throughout this section:

\begin{assumption}\label{a:blownup}
    Let $n \ge 1$ and $M \in \R^{n\times n}$ be a $\Lambda_0$-elliptic matrix. Let $u \in W^{1,2}_{\loc}\cap C(\R^n)$ be a weak solution to
    \begin{equation}\label{e:B-sol}
        L^{M} u = - \diverg (M(u)\nabla u) = 0 \qquad \text{on $\R^n$}\,,
    \end{equation}
    where
    \[
        M(u) := \id \mathbf{1}_{\{u>0\}} + M \mathbf{1}_{\{u\leq 0\}}\,.
    \]
\end{assumption}

The main result of this section can be loosely summarized as follows: if $n\geq 3$, $\omega\in \Dcal(\Id,M)$, and $\omega$ is sufficiently close to flat at infinity, then in fact $\omega \in \cF$. This will follow by showing a Liouville-type theorem for the difference of Green's functions $u^\pm$ with poles at infinity for $(\Omega^\pm(u),\omega,L^M)$, which we will see form a viscosity solution of an appropriate two-phase free-boundary problem; see Corollary \ref{c:flat-fb} for a precise statement of the result.

\subsection{Weak solutions are viscosity solutions}\label{ss:viscosity-sol}

\begin{definition} \label{d:viscosity}
    We say that a continuous function $u$ is a viscosity sub-solution to our free boundary problem $L^{M}(\cdot) = 0$ if:
    \begin{enumerate}
        \item $\Delta u \ge 0$ in $\Omega^{+}(u) = \{ u > 0\}$ and $L_{M}u \ge 0$ in $\Omega^{-}(u) = \{ u < 0 \}$.
        \item If $B(y,\rho) \subset \Omega^{+}(u)$ and $x_{0} \in \partial B(y,\rho) \cap \partial \Omega^{+}(u)$, then for the inward pointing unit normal $\nu$ of $\partial B(y,\rho)$ at $x_{0}$,
        $$
            u(x) \ge \alpha (\nu \cdot(x-x_{0}))^{+} - \beta (\nu (x-x_{0}))^{-} + o(|x-x_0|)
        $$
        for some $\beta \ge 0$ and $\alpha \ge {\big(\sum_{i,j=1}^{n} \nu_{i} \nu_{j} M^{ij}\big) \beta}$ where $M^{ij}$ are the coefficients of $M$.
    \end{enumerate}

    A super-solution is defined similarly, with all the inequalities above reversed and the ball $B(y,\rho) \subset \Omega^{-}(u)$.

    A solution is any continuous function that is both a sub- and a super-solution.
\end{definition}    

\begin{remark}
    {We are considering the very same two-operator free-boundary problem and notion of solution as that in \cite{Feldman,AM} (see also \cite{Cerutti-Ferrari-Salsa,Ferrari-Salsa} for a more general variable coefficient version). Indeed, the crucial step herein is to verify that the difference of blown up Greens functions are solutions to this free-boundary problem and satisfy the $\eps$-monotonicity property required to kick-start the improvement of flatness mechanism. Thereafter, we are simply invoking the strategy presented in the aforementioned works, but we provide the key steps and relevant intermediate result statements herein for the purpose of clarity. Our exposition follows that in these references closely.}
\end{remark}
    
We will frequently use the notion of \emph{two-plane solutions}:
\begin{definition}
    We refer to a function $P(x) = \alpha(\nu\cdot x)^+ - \beta(\nu\cdot x)^-$ for some choice of $\nu\in \Sbb^{n-1}$, $\alpha,\beta\geq 0$ as a two-plane solution. Given $M$ as in Assumption \ref{a:blownup}, if $\alpha, \beta$ satisfy the property (2) of Definition \ref{d:viscosity} for this choice of $M$, we refer to $P$ as a two-plane sub-solution of $L^M (\cdot) = 0$. We analogously define two-plane super-solutions of $L^M (\cdot) =0$, and we refer to $P$ as a two-plane solution of $L^M (\cdot)=0$ if it is both a two-plane sub- and super-solution. {There will often be no relationship assumed between $\alpha$ and $\beta$. We will thus often simply talk about two-plane solutions without specifying any coefficients $\alpha,\beta$.}
\end{definition}

We begin with the following lemma.

{
\begin{lemma} \label{l:twoplane}
    Suppose that $M$ and $u$ are as in Assumption \ref{a:blownup}. If $B(y,\rho) \subset \Omega^{-}(u)$ or $B(y,\rho) \subset \Omega^{+}(u)$ with $x_{0} \in \partial B(y,\rho) \cap \{u=0\}$, then there exists a unique two-plane solution $P= P_{x_{0}}$ so that
    \begin{equation} \label{e:am36}
        u(x) - P(x-x_{0})= o(|x-x_{0}|).
    \end{equation}
    In particular, the rescalings
    $$
        u_{r}(\cdot) := \frac{u(r \cdot + x_{0})}{r}
    $$
    converge locally uniformly to $P$. Moreover, the following three properties hold:
    \begin{enumerate}
        \item For all $N > 0$,
    \begin{equation} \label{e:l1setconv}
        \frac{|(\Omega^{\pm}(u) \Delta \Omega^{\pm}(P)) \cap B(x_{0},rN)|}{|B(x_{0},rN)|} \xrightarrow{r \to 0} 0.
    \end{equation}
        \item $u_{r} \rightarrow{P}$ in $L^{2}_{\loc}(\R^{n})$ and $\nabla u_{r} \rightarrow \nabla P$ weakly in $L^{2}_{\loc}(\R^{n})$
        \item  $P$ satisfies $L^M P = 0$ in a weak sense.
    \end{enumerate}

\end{lemma}

\begin{remark}
    The existence of a unique two-plane solution at points with touching balls is demonstrated in \cite[Lemma 3.6]{AM}. {However, one should be careful about concluding the uniqueness of the two-plane solution from the non-degeneracy of the Alt-Caffarelli-Friedman density alone. Indeed the uniqueness may fail, see \cite{Allen-Kriv}. However, in our setup, this uniqueness is guaranteed; the presence of a touching ball prevents the zero-set of the blow-up from rotating, while the uniqueness of the slopes is guaranteed by \cite[Lemma 11.17]{CaffSalsa}.}
\end{remark}

\begin{proof}
    The asymptotic development \eqref{e:am36} is the conclusion of \cite[Lemma 3.6]{AM}, {see also \cite[Lemma 11.17]{CaffSalsa}. Note that the conclusion of \cite[Lemma 3.6]{AM} is indeed for weak solutions of \eqref{e:B-sol}; we have not yet established that $u$ is a viscosity solution.} This immediately implies the rescalings $u_{r}$ converge locally uniformly to $P$. To prove \eqref{e:l1setconv} we first suppose without loss of generality that $x_{0} = 0$ and observe that $\Omega^{\pm}(P)$ are half spaces and hence invariant under scaling. Then for any fixed $N>0$ we have
    \begin{align*}
        \frac{|(\Omega^{\pm}(u) \Delta \Omega^{\pm}(P)) \cap B_{rN}|}{|B_{rN}|}  & = \frac{|(\Omega^{\pm}(u_{rN}) \Delta \Omega^{\pm}(P)) \cap B_{1}|}{|B_{1}|} \\
        & \le \int_{B_{1}} |\eye_{\Omega^+(u_{rN})} - \eye_{\Omega^+(P)}| + |\eye_{\Omega^-(u_{rN})} - \eye_{\Omega^-(P)}| \,,
    \end{align*}
    which is seen to go to zero by the dominated convergence theorem and the fact that $u_{r} \to P$ locally uniformly, so in particular pointwise, which in turn implies that $\Omega^\pm(u_{r})$ converges pointwise to $\Omega^\pm(P)$.

    To prove (2) we must first show that for any $N >1$ and  all $r$ sufficiently small, $\| u_{r}\|_{W^{1,2}(B_{N})} \le C < \infty$ for some constant $C$. Indeed, this suffices since then for any sequence $\{u_{r_{i}}\}$ there exists a subsequence $\{u_{r_{i_{j}}}\}$ and a function $g \in W^{1,2}(B_{N})$ (which may a priori depend on the subsequence) so that $u_{r_{i_{j}}} \xrightarrow{L^{2}} g$ and $\nabla u_{r_{i_{j}}} \xrightharpoonup{L^{2}} \nabla g$. However, by \eqref{e:am36} we then deduce that $g = P$. Since $N$ is arbitrary, this will prove (2).

    To verify the $W^{1,2}(B_{N})$ bound, we first use the Caccioppoli inequality with a function $\varphi \in W^{1,\infty}_{0}(B_{N+1})$ so that $\varphi \equiv 1$ on $B_{N}$ and $\|\nabla \varphi\|_{L^\infty(B_{N+1})} \lesssim 1$ to deduce
    \begin{align*}
    \int_{B_{N}} |\nabla u_{r}|^{2} \le \int |\nabla u_{r}|^{2} \varphi^{2} \lesssim_{\Lambda_{0}} \int |\nabla \varphi|^{2} u^{2} \lesssim \int_{B_{N+1} } u_{r}^{2}.
    \end{align*}
    By \eqref{e:am36}, it follows that for $r$ small enough that on $B_{N+1}$ we have $|u_{r}| \lesssim N+1$ with suppressed constants independent of $N$ so that
    $$
        \int_{B_{N}} |\nabla u_{r}|^{2} \lesssim (N+1)^{2} |B_{N+1}|.
    $$
    In particular again using \eqref{e:am36} implies that
    $$
        \|u_{r}\|_{W^{1,2}(B_{N})}^{2} = \|u_{r}\|_{L^{2}(B_{N})}^{2} + \|\nabla u_{r}\|_{L^{2}(B_{N})}^{2} \lesssim {N^{n+2} + (N+1)^{n+2}} < \infty
    $$
    for all $r$ sufficiently small, completing the proof of (2).

    We next claim that (3) follows from (1) and (2). Indeed, for $\varphi \in W^{1,{2}}_{0}(B_{N})$, (2) implies
    \begin{align} \label{e:3.1}
        \int & \langle M(P) \nabla P, \nabla \varphi \rangle \overset{(2)}{=} \lim_{r \to 0} \int \langle M(P) \nabla u_{r}, \nabla \varphi \rangle\,.
    \end{align}
    On the other hand, 
    \begin{align*}
        \limsup_{r \to 0} & \bigg| \int \Langle M(P) \nabla u_{r}, \nabla \varphi \Rangle - \int \Langle M(u_{r}) \nabla u_{r}, \nabla \varphi \Rangle \bigg|  \nonumber \\
        & \le \limsup_{r \to 0} \int_{( \Omega^{+}(u_{r}) \Delta \Omega^{+}(P)) \cap B_{N}} |M(u_{r}) - M(P)| |\nabla u_{r}| |\nabla \varphi| \nonumber \\
        & \lesssim_{\Lambda_{0}} {\limsup_{r \to 0} \left( \int_{ (\Omega^{+}(u_{r}) \Delta \Omega^{+}(P)) \cap B_{N}} |\nabla \varphi|^{2} \right)^{1/2}}  \|\nabla u_{r}\|_{L^{2}(B_{N})} = 0
    \end{align*}
    where the final equality uses (1), (2), Cauchy-Schwarz, {and the fact that $\nabla \varphi \in L^{2}$, so that in particular $|\nabla \varphi|^{2} \, dx$ defines a measure which is mutually absolute continuous with respect to the Lebesgue measure.} Combining this with \eqref{e:3.1} verifies (3).
\end{proof}
}

\begin{lemma} \label{l:weak2visc}
    Let $M$ and $u$ be as in Assumption \ref{a:blownup}. Then $u$ satisfies $L^M u=0$ in a viscosity sense.    
\end{lemma}

In \cite{AM} it is claimed without proof that \eqref{e:am36} implies weak solutions are viscosity solutions. While the fact that the first condition in Definition \ref{d:viscosity} is satisfied is classical, we fill in the details to prove the second condition in Definition \ref{d:viscosity} also follows from Lemma \ref{l:twoplane} as claimed.

\begin{proof}[Proof of Lemma \ref{l:weak2visc}]
    By \eqref{e:am36} from Lemma \ref{l:twoplane}, we know that if there exists a touching ball $B(y,\rho)$ contained in either $\Omega^+$ or $\Omega^-$ and $\partial B(y,\rho) \cap \{u=0\} \ni x_{0}$ then if $\nu$ is the inward pointing normal direction to $B(y,\rho)$ at $x_{0}$, there exists $\alpha, \beta$ so that
    $$
        u(x) = \alpha ((x-x_{0}) \cdot \nu)^{+} + \beta ((x-x_{0}) \cdot \nu)^{-} + o(|x-x_{0}|) =: P(x-x_0) + o(|x-x_{0}|).
    $$
    Without loss of generality, suppose $x_{0} = 0$. It remains to check that the relationship between $\alpha$ and $\beta$.  Let $v_{1}(x) =  \alpha (x \cdot \nu)^{+}$ and $v_{2}(x) = \beta (x \cdot M \nu)^{-}$ so that $M(P) \nabla P = \mathbf{1}_{P> 0} \nabla v_{1} + \mathbf{1}_{P\leq 0}  \nabla v_{2}$. Fix $0 \le \varphi \in W^{1,2}_{0}(B_{r})$. Then, because $P$ satisfies $L^M P=0$ in a weak sense, by Lemma \ref{l:twoplane} (3) we have 
\begin{align*}
    0 &\le \int_{B_{r}} M(P) \nabla P \cdot \nabla \varphi \, dx \\
    &= \int_{B_{r} \cap \Omega^+(P)} \nabla v_{1} \cdot \nabla \varphi \, dx + \int_{B_{r} \cap \Omega^-(P) }  \nabla v_{2} \cdot \nabla \varphi \, dx \\
    & = - \int_{B_{r} \cap \Omega^+(P)} \varphi \Delta v_{1} \, dx - \int_{B_{r} \cap \Omega^-(P)} \varphi \Delta v_{2} \, dx \\
    &+ \int_{\partial(B_{r} \cap \Omega^+(P))} \varphi \partial_{\nu} v_{1} d \cH^{n-1} + \int_{\partial(B_{r} \cap \Omega^-(P))} \varphi \partial_{\nu} v_{2} d \cH^{n-1} \\
    & = \int_{ \{P=0\} \cap B_{r}} \varphi \left( \alpha \nu - M \beta \nu \right) \cdot \nu d \cH^{n-1} \,.
\end{align*}
Here, the final equality uses $\varphi|_{\partial B_{r}} = 0$ (in the Sobolev trace sense) and $\Delta v_{1}= \Delta v_{2} = 0$. Since $\varphi \ge 0$ was arbitrary in $W^{1,2}_{0}(B_{r})$, this in turn implies ${\alpha = \alpha \nu^t \nu} \ge \beta \nu^{t}M\nu$ confirming the desired inequality. Repeating the argument with a touching ball from the other side, and using that $P$ is a weak super-solution of $L^M(\cdot)=0$ {confirms $\alpha \le \beta \nu^{t} M \nu$ verifying that $P$ is a solution to $L^{M}(\cdot) = 0$ in the viscosity sense.}
\end{proof}

\subsection{\texorpdfstring{$\eps$-monotonicity}{epsilon-monotonicity}  for blow-ups}\label{ss:eps-mon}
In this section we will work under the following assumption, {which we know is satisfied by the $\Lambda$-tangents in $F^{*}$} thanks to Theorem \ref{t:blowups}, Corollary \ref{c:Lambda-tangent-FB}, and Lemma \ref{l:weak2visc}.
\begin{assumption}\label{a:eps-mon}
     Assume $n \ge 3$ and that $\omega \in \cD(\id, M)$ for some $\Lambda_{0}$-elliptic matrix $M$. In particular, $0 \in \spt \omega$ and $L^{M}u=0$ on $\R^{n}$ in a viscosity sense, where $u = u^{+} - u^{-}$, and $u^{\pm}$ are Green's functions with poles at infinity for $(\Omega^\pm(u), \omega, L^M)$.
\end{assumption}

We begin by showing that closeness to flat in the sense of measures $\omega\in \Dcal(\Id,M)$ guarantees closeness to two-plane solutions for the associated difference of Green's functions.
\begin{lemma}\label{l:compactness}
    Let $\Lambda_{0}, r_{0} > 0$ be fixed. For any $\eta > 0$ there exists $\delta = \delta(n, \Lambda_{0}, \eta,r_{0})>0$ such that the following holds. Suppose $M$, $\omega$ and $u$ are as in Assumption \ref{a:eps-mon}. If for all $r \ge r_{0}$, 
    $$
        d_{r}(\omega, \cF) \le \delta
    $$
    then
    $$  
        \inf_{P \text{ two-plane solutions}} \|u-P\|_{L^{\infty}(B_{1})} \le \eta \, \omega(B_{1}) \,,
    $$
    where we are taking two-plane solutions $P$ of $L^M (\cdot) =0$.
\end{lemma}

\begin{proof}
    Suppose that the conclusion of the lemma fails.  Then, there exists a sequence of measures $\widetilde{\omega}_{k} \in \cD(\id,M)$, which we immediately replace with the measures $\omega_{k} = \widetilde{\omega_{k}}(B_{1})^{-1} \widetilde{\omega_{k}}$, so that
    \begin{equation} \label{e:approachingflat}
        d_{r}(\omega_{k},\cF) \le 2^{-k} \qquad \forall r \ge r_{0}\,,
    \end{equation}
    but 
    \begin{equation} \label{e:not2ps}
        \inf_{P \text{ two plane solutions}} \|u_{k}-P\|_{L^{\infty}(B_{1})} \ge \eta_{0} \,
    \end{equation}
    where $u_{k}^{\pm}$ are the Green's functions with pole at infinity guaranteed in Definition \ref{d:cda1a2}. 
    
     Since $\omega_k(B_1)=1$ for each $k$, Lemma \ref{l:doubling} in turn implies 
     \begin{equation} \label{e:doubling}
     \omega_{k}(B_{2^{N}}) \lesssim C^{N}
     \end{equation} where $C$ is the doubling constant from Lemma \ref{l:doubling} and the suppressed constants are independent of $k$. Thus, there exists a convergent subsequence (not relabeled) with $\omega_{k} \xrightharpoonup{*} \omega_{\infty}$. Since $d_{r}(\cdot, \cF)$ is continuous with respect to weak-$*$ convergence, it follows from \eqref{e:approachingflat} that $d_{r}(\omega_{\infty},\cF) = 0$ for all $r \ge r_{0}$. This guarantees that $\omega_{\infty} \in \cF$. On the other hand, since \eqref{e:doubling} holds for all $k \in \N$, Lemma \ref{l:cfms} implies $|u_{k}^{\pm}(x_{0,N})| \lesssim_{N} 1$ for all $k$. By Harnack chains this implies $\|u_{k}^{\pm}\|_{L^{2}(B(0,N))} \lesssim_{N} 1$ for all $k$ and all $N$. In particular, De Giorgi-Nash-Moser implies $\{u_{k}\}$ are locally uniformly equicontinuous and we assumed $u_{k}(0) = 0$. So, Arzela-Ascoli's theorem produces some function $u_{\infty}$ that is the locally uniform limit of $u_{k}$, up to another subsequence that we do not relabel.

    As in Theorem \ref{t:blowups}, the relationship between $u_{k}$ and $\omega_{k}$ passes to limit, that is we also know
    $$
        \int \varphi \, d \omega_{\infty} = \int \langle M(u_{\infty}) \nabla u_{\infty}, \nabla \varphi \rangle \qquad \forall \varphi \in C_{c}^{\infty}(\R^{n})
    $$
    and in particular that $L^{M}u_{\infty} = 0$. Writing $u_{\infty} = u_{\infty}^{+} - u_{\infty}^{-}$. Since $\omega_{\infty} \in \cF$ we have that $u_{\infty}^{+}$ is a positive harmonic function on a half-space, vanishing at the boundary. This implies $u_{\infty}^{+} = \alpha (x \cdot \nu)^{+}$, where $\nu$ is the inward pointing normal to $\Omega^{+}(u)$ (see e.g. \cite{BB}). Similarly, one confirms that $u_{\infty}^{-}$ is of the form $\beta (x \cdot \nu)^{-}$ by a rigid change of variables. In particular, there is a two-plane solution $P$ so that $u_\infty = P$.  But since $u_{k}$ converges to $u_{\infty}$ uniformly on $B_{1}$, this contradicts \eqref{e:not2ps}.
\end{proof}

Recalling that we aim to use the techniques of \cite{Caff1-Lip-implies-reg,Caff2-flat-implies-Lip,Caff3} (see also \cite{AM,Feldman}), we are now in a position to derive $\eps$-monotonicity of solutions to our limiting free boundary problem.
\begin{definition}[$\eps$-monotonicity]
	Let $\eps>0$ and $U\subset \R^n$ open. We say that $u\in C(U)$ is $\eps$-monotone in direction $v\in \Sbb^{n-1}$ on $U$ if for every $\tilde\eps \geq\eps$ we have $u(x- \tilde\eps v) \leq u(x)$ for every $x, {x- \tilde\eps v}\in U$.
	
	We say that $u$ is $\eps$-monotone in the cone $\Gamma(\theta,e) := \{v \in \Sbb^{n-1}: v\cdot e \geq \cos\theta \}$ if $u$ is $\eps$-monotone in direction $v$ for every $v\in \Gamma(\theta,e)$.
\end{definition}

\begin{lemma}\label{l:eps-mon}
    Let $\Lambda_{0}, r_{0} > 0$ and let $\omega, u,M$ be as in Assumption \ref{a:eps-mon}. For every $\eps \in (0,1)$ and $0 < \eta < \frac{\eps}{4C}$ for a constant $C$ depending only on the NTA constants the following holds:

    There exists $0 < \delta_{0}= \delta_{0}(n, \Lambda_{0}, C_2,R_2, \eta)$ given by the conclusion of Lemma \ref{l:compactness} such that if $d_{r}(\omega, \cF) \le \delta_{0}$  for every $r \ge r_{0}$ , then there is a direction $\nu$ so that $u$ is $\eps$-monotone in $B_{1/2}$ in the cone $\Gamma(\theta,\nu)$ for any $\theta$ so that $\cos(\theta) \ge \frac{4 \eta C}{\eps}$.
\end{lemma}

\begin{remark}
    In order to derive $\eps$-monotonicity, we require a certain nondegeneracy of $u$, {in the sense that $u$ needs to be much closer to an optimal two-plane solution \emph{relative to the slope of the two-plane solution}. This is precisely the conclusion of Lemma \ref{l:compactness}, since $\omega(B_1)$ is comparable to the slope of an optimal solution $P$ in light of the CFMS estimates of Lemma \ref{l:cfms}.} Notice that in \cite[Lemma 6.3]{AM}, the authors guarantee the conclusion of Lemma \ref{l:compactness} by considering a rescaling of $u$ around a point in $p \in \{u=0\}$, at a sufficiently small scale (depending on $p$), in order to satisfy the hypothesis of being close to a two-plane solution. This is because they are restricting themselves to points where two-plane blow-ups exist, which we cannot do in the present setting. This explains why we require additional assumptions like Assumption \ref{a:eps-mon} or Assumption \ref{a:blownup} plus NTA; these provide the setup in order to guarantee that the conclusion of Lemma \ref{l:compactness} can be turned into a closeness-proportional-to-slope condition as described above, which is crucial for obtaining $\eps$-monotonicity. {It is also important for us that the dependency of the proportionality $\eta$ required in order to guarantee the $\eps$-monotonicity is \emph{uniform}, so that we may repeat this procedure around different center points with uniform dependencies. This is because we are striving for a \emph{globally uniform} Lipschitz constant bound on $\{u=0\}$, which in turn gives flatness when blowing down, thanks to the $C^{1,\alpha}$-estimate that it induces}.
\end{remark}

\begin{proof}
    Fix $\eps \in (0,1)$ and $0 < \eta  < \frac{\eps}{4 C}$. By Lemma \ref{l:compactness}, there exists $\delta_0 = \delta_0(n, \Lambda_{0},\eta,C_2,R_2) > 0$ such that if $d_{r}(\omega,\cF) \le \delta_0$ for all $r \ge 1$ then there exists a two-plane solution $P = \alpha (x \cdot \nu)^{+} - \beta (x \cdot \nu)^{-}$ so that
    \begin{equation} \label{e:almost2ps}
        \|u-P\|_{L^{\infty}(B_{1})} \le \eta \,  \omega(B_{1}).
    \end{equation}

    We claim this implies $u$ is $\eps$-monotone in $\Gamma(\theta,\nu)$ for $\theta$ such that $\cos(\theta) \ge \frac{4 \eta C}{\eps}$. Note that such $\theta$ exists since $\eta < \frac{\eps}{4 C}$.

    Indeed, if $x^{\pm}_{0,1}$ denote corkscrew points for $\Delta(0,1)\subset \partial \Omega(u)$, then Lemma \ref{l:cfms} implies that $u^{\pm}(x_{0,1}^{\pm}) \ge C^{-1} \omega(B_{1})$ for some constant $C>0$ depending only on $n$ and the NTA constants $C_2,R_2$. Since in particular $\eta < C^{-1}/2$, we deduce from \eqref{e:almost2ps}, the fact that $x_{0,1}^\pm\in B_1$, and the fact that $P$ is a two plane solution with $\{P=0\} = \nu^{\perp}$ that
    $$
        |P(\pm \nu)|  \ge \Big|P\big(x^{\pm}_{0,1}\big)\Big| \ge  \Big|u^{\pm}(x^{\pm}_{0,1})\Big| - \Big|u^{\pm}\big(x^{\pm}_{0,1}\big) - P(\pm x_{0,1}^{\pm})\Big| \ge C^{-1} \omega(B_{1})/2= : m_{0} .
    $$
    In particular, since $P$ is a two-plane solution this implies that
    \begin{equation} \label{e:nondegslope}
        P(x+t\nu) \ge m_{0} t + P(x) \qquad \forall x \in \R^{n} \quad \forall t >0.
    \end{equation}

    Combining \eqref{e:almost2ps} with \eqref{e:nondegslope}, it follows if $x \in B_{1/2}$,
    $$
        u(x) \le P(x) + \eta \omega(B_{1}) \le P(x+t\nu) - m_{0}t + \eta \omega(B_{1}) \le u(x+t\nu) - m_{0} t + 2 \eta \omega(B_{1}).
    $$
    so that $u(x + t \nu) \ge u(x)$ if $t \ge \frac{2 \eta \omega(B_{1})}{m_{0}} = 4 \eta C$. Since $\eps \ge 4 \eta C$, this demonstrates that $u$ is $\eps$-monotone in direction $\nu$. Because $P$ is a two-plane solution, one could re-write \eqref{e:nondegslope} as saying that 
    $$
        P(x+te) \ge m_{0} \langle e,\nu \rangle t + P(x)\,,
    $$
    so an analogous computation verifies $\eps$-monotonicity for $\eps\geq \frac{4 \eta}{C \cos(\theta)}$ in $\Gamma(\theta,\nu)$. Recalling the definition of $m_0$, the proof is complete.
\end{proof}

\subsection{\texorpdfstring{$\eps$-monotonicity implies Lipschitz boundary}{epsilon-monotonicity implies Lipschitz boundary} }\label{ss:eps-mon-implies-Lip}
Armed with Lemma \ref{l:eps-mon}, we are now in a position to follow the arguments of \cite{Caff2-flat-implies-Lip, Wang-nonlinear-flat-implies-Lip,AM} (see also \cite{Ferrari-Salsa}) to obtain (in Lemma \ref{t:Lip}) local Lipschitz regularity of our free boundary, which will in turn imply (in Theorem \ref{t:Lip-implies-C1alpha}) local $C^{1,\alpha}$ regularity, see \cite{Caff1-Lip-implies-reg, Feldman,Cerutti-Ferrari-Salsa}. This will then allow us to obtain a global Liouville-type result (see Corollary \ref{c:flat-fb}). The reader may observe that although the results in \cite{Caff2-flat-implies-Lip} are stated for a rather specific two-phase free boundary problem where the operator on each side is the Laplacian, with a transmission condition across the boundary, the results hold for more general free-boundary problems, in particular the one of the setting herein, provided that
\begin{itemize}
    \item[(i)] One is considering viscosity solutions to a two-phase free-boundary problem with a transmission condition at the boundary that satisfies a suitable monotonicity and regularity condition; see e.g. \cite[Theorem 1]{Caff2-flat-implies-Lip}.
    \item[(ii)] The operator for the interior constraint in the free-boundary problem satisfies a suitable comparison principle, including at the boundary (boundary Harnack). Note that one does not require the operator to be linear in general, see e.g. \cite{Wang-nonlinear-flat-implies-Lip}.
    \item[(iii)] The solution to our free boundary problem is $\eps$-monotone in some initial cone of directions.
\end{itemize}
We refer the reader to \cite[Chapters 4 \& 5]{CaffSalsa} for a good presentation of the ideas in such arguments.

\begin{remark}
    Although the more recent and flexible methods of De Silva \cite{DeSilva} (see also \cite{DFS-2phase}) are now commonly used in place of Caffarelli's original techniques to establish local $C^{1,\alpha}$ regularity for sufficiently flat free boundaries, it seems that, {aside from the fact that we do not have the stronger form of non-degeneracy required to implement De Silva's approach (see (H2) in \cite{DFS-2phase}}, the methods of Caffarelli are more well-suited to the kind of global rigidity (exact flatness) statement we are seeking herein, due to the fact that the argument passes through an intermediate Lipschitz bound.
\end{remark}

The main result of this section is the following, which is the conclusion of \cite[Proposition 6.11, Theorem 6.12]{AM} (see also \cite[Lemma 5.7]{CaffSalsa}).

\begin{theorem}\label{t:Lip}
    Suppose that $u$, $\omega$, $M$ satisfy Assumption \ref{a:eps-mon} and $r_0>0$ is fixed, there exists $\delta_1=\delta_1(n,\Lambda_0,C_2,R_2)>0$ such that if $d_r(\omega,\Fcal)\leq \delta_1$ for every $r\geq r_0$, then $u$ is fully monotone in $\Ccal_{1/4}:= (B_{1/4} \cap \nu^\perp)\times (-\tfrac{1}{4}, \frac{1}{4})$ in the cone $\Gamma(\theta_1,\nu)$ for some $\theta_1(\theta_0,\eps)$. In particular, $\{u=0\}\cap \Ccal_{1/4}$ is a Lipschitz graph with Lipschitz constant $\kappa=\kappa(n,\Lambda_0,C_2,R_2)>0$, over $\nu^\perp$. Here, $\nu^\perp$ denotes the $(n-1)$-dimensional linear subspace of $\R^n$ orthogonal to $\nu$.
\end{theorem}

We will not provide all of the details of the proof of Theorem \ref{t:Lip}, since they are contained in \cite[Section 6]{AM}. Nevertheless, for the purpose of clarity, we will provide a breakdown of the important intermediate results leading towards its conclusion, and emphasize the parts for which the arguments differ to their classical counterparts in \cite{Caff2-flat-implies-Lip}. {We will also demonstrate how all the intermediate results that we state herein are then combined in order to conclude Theorem \ref{t:Lip}, for the benefit of the reader.

Let us begin with stating the key underlying improvement of $\eps$-monotonicity (cf. \cite[Proposition 6.11]{AM} and \cite[Lemma 5.7]{CaffSalsa}) underlying Theorem \ref{t:Lip}, which is the ultimate goal of the scheme of Caffarelli.  The moral is that one may gain $\eps$-monotonicity (by a proportional factor) at the cost of shrinking the angle of the cone in which it holds (as well as the domain of the function) by an \emph{additive} constant that is a fractional power of $\eps$.

\begin{proposition}\label{p:impr-of-mon}
    Suppose that $u$ and $M$ satisfy Assumption \ref{a:blownup} with $u(0)=0$, let $\tfrac{\pi}{4} < \theta_0 \leq \theta \leq \tfrac{\pi}{2}$, and let $\nu\in \Sbb^{n-1}$. There exist $\eps_1=\eps_1(n,\Lambda_0,\theta_0)>0$, $\lambda=\lambda(n,\Lambda_0,\theta_0)\in (0,1)$, and $c_0=c_0(n,\Lambda_0,\theta_0)>0$ such that the following holds for any $\eps\in (0,\eps_1)$. If $u$ is $\eps$-monotone in $B_{1/2}$ in the cone $\Gamma(\theta,\nu)$, then $u$ is $\lambda\eps$-monotone in the cone $\Gamma(\theta- c_0\eps^{1/4},\nu)$ in $\Ccal_{1/2-c_0 \eps^{1/8}}$.
\end{proposition}

Let us first demonstrate how to deduce Theorem \ref{t:Lip} from Proposition \ref{p:impr-of-mon}.

\begin{proof}[Proof of Theorem \ref{t:Lip}]
    By iterating Proposition \ref{p:impr-of-mon}, provided that $u$ is $\eps$-monotone in $B_{1/2}$ in some cone of directions $\Gamma(\theta,\nu)$, we obtain a sequence of cylindrical domains
    \[
        \Ccal_{\rho_k} = (B_{\rho_k} \cap \nu^\perp)\times (-\rho_k, \rho_k)\,,
    \]
    with
    \[
        \rho_k := \frac{1}{2} - \sum_{j=1}^k c_0 (\lambda^j\eps)^{1/8}\,,
    \]
    such that $u$ is $\lambda^k \eps$-monotone in the cone of directions $\Gamma(\theta_k, \nu)$, where
    \[
        \theta_k := \theta - \sum_{j=1}^k c_0 (\lambda^j\eps)^{1/4}\,.
    \]
    This concludes the proof, for $\eps\in(0,\eps_1)$ chosen small enough such that 
    \[
        \sum_{j=1}^k c_0 (\lambda^j\eps)^{1/8} \leq \frac{1}{4}\,,
    \]
    and $\theta_1 = \lim_{k \to \infty} \theta_k$. This yields the claimed full monotonicity in $\Gamma(\theta_1, \nu)$, conditional on the initial $\eps$-monotonicity.

    On the other hand, the initial $\eps$-monotonicity in some cone $\Gamma(\theta,\nu)$ follows immediately from Lemma \ref{l:eps-mon}. It is then elementary to verify that the Lipschitz graphicality of the free boundary follows; we leave this to the reader.
\end{proof}
}

We now proceed to outline the key intemediate results going into the proof of Proposition \ref{p:impr-of-mon}. The starting point is the following lemma, which allows one to improve $\eps$-monotonicity to full monotonicity for a barrier solution associated to $u$ in balls with radius given by a function, in a cone whose angle is determined by the gradient of this function.

\begin{lemma}{\cite[Lemma 2]{Caff2-flat-implies-Lip}}\label{l:eps-mon-barrier-mon}
    Let $\vphi \in C^2(B_{1/2};(0,\rho_0])$ with $\rho_0 < \frac{1}{2}$ and suppose that $u\in C(B_{1/2})$ is $\eps$-monotone in a cone $\Gamma(\theta,\nu)$. For $s \leq \frac{1}{2} - \|\vphi\|_{L^\infty}$, consider the function $v$ on $B_s$ given by
    \begin{equation}\label{e:barrier}
        v(x) := \sup_{B(x,\vphi(x))} u\,.
    \end{equation}
    Assume in addition that $\tilde\theta$ satisfies
    \[
        \sin\tilde\theta \leq \frac{1}{1+|\nabla\vphi(x)|} \left(\sin \theta - \frac{\eps}{2\vphi(x)}\cos^2\theta - |\nabla \vphi(x)| \right)\qquad \text{for every $x\in B_{1/2}$}\,.
    \]
    Then $v$ is monotone in the cone $\Gamma(\tilde\theta,\nu)$. In particular, in this case the level sets of $v$ are Lipschitz graphs with Lipschitz constant $\bar{L} \leq \cot\tilde\theta$ over hyperplanes orthogonal to $\nu$ in $\R^n$.
\end{lemma}
Note that Lemma \ref{l:eps-mon-barrier-mon} does not require any assumptions on $u$ other than $\eps$-monotonicity.

An additional key observation about the function $v$ in \eqref{e:barrier}, as demonstrated in \cite{Feldman}, is that it is a subsolution of our free boundary problem away from $\{v=0\}$.

\begin{lemma}{\cite[{Lemma 7}]{Feldman}}\label{l:subsol}
    Let $u$, $M$, and $L^M$ be as in Assumption \ref{a:blownup}. There exists $C=C_{\ref{l:subsol}}(n,\Lambda_0)>0$ large enough such that the following holds. Suppose that $\vphi \in C^2(B_{1/2};(0,\tfrac{1}{2}))$ satisfies
    \[
        \vphi L^M \vphi \geq C |\nabla \vphi|^2\,.
    \]
    Then $v$ as in \eqref{e:barrier} for the function $u$ satisfies $L^M v \leq 0$ in $B_{1/2 - \|\vphi\|_{L^\infty}}\setminus \{v=0\}$.
\end{lemma}
Note that although in \cite{Feldman}, $v$ is defined as a supremum over $\partial B(x,\vphi(x))$ in place of $B(x,\vphi(x))$, in Lemma \ref{l:subsol} this is equivalent to our definition in light of the maximum principle.
The next step is to establish existence of a preliminary 1-parameter family of functions $\{\vphi_t\}_t$ that will satisfy suitable estimates in order for us to use them in Lemma \ref{l:eps-mon-barrier-mon}. This will provide us with something that will be almost a sub-solution to our free-boundary problem.

\begin{lemma}\label{l:1-param-for-barrier}
    Let $\pi\subset \R^n$ be a hyperplane. Let $G$ be the graph $\{(x',f(x')): x' \in \pi\cap B_{1/2}\}$ of a Lipschitz function $f$ with $f(0)=0$ and Lipschitz constant $\bar L$, and let $\Ccal = (\pi\cap B_{1/2}) \times [-2\bar L, 2\bar L]$. Then {there exists a constant $C=C_{\ref{l:1-param-for-barrier}}(n,\Lambda_0) \geq C_{\ref{l:subsol}}$ such that} for every $\delta \in (0,\tfrac{1}{2})$, there exists a family $\{\vphi_t\}_{t\in [0,1]} \subset C^2(B_{1/2})$ satisfying 
    \begin{itemize}
        \item[(i)] $1\leq \vphi_t \leq 1+t$;
        \item[(ii)] $\vphi_t L^M \vphi_t \geq C |\nabla \vphi_t|^2$;
        \item[(iii)] $\vphi_t \simeq 1$ on $G_\delta := \{\dist(\cdot, G\cap \partial \Ccal) < \delta)\}$;
        \item[(iv)] $\vphi_t(x) \geq 1 + t \left[1- C\delta \dist(x,\partial\Ccal)^{-2}\right]$ on $\{\dist(\cdot, \partial\Ccal) > \delta\}$;
        \item[(v)] $|\nabla\vphi_t| \leq \frac{Ct}{\delta}$.
    \end{itemize}
\end{lemma}
The proof of Lemma \ref{l:1-param-for-barrier} can be found in \cite[{Lemma 3}]{Wang-nonlinear-flat-implies-Lip}, combined with the observation that $L^M \vphi_t \geq \Mcal^-(D^2\vphi_t, \Lambda_0^{-1},\Lambda_0)$, where $\Mcal^-(D^2\vphi_t, \Lambda_0^{-1},\Lambda_0)$ denotes the Pucci minimal operator as defined in e.g. \cite[{Section 2.2}]{CaffCabre}.

For $u$ as in Lemma \ref{l:eps-mon-barrier-mon} and $\{\vphi_t\}$ as in Lemma \ref{l:1-param-for-barrier}, the function
\begin{equation}\label{e:subsol}
    v_t(x) := \sup_{B_{\sigma\vphi_t(x)}(x)} u\,,
\end{equation}
with $\frac{\eps}{2} <\sigma < 2\eps$ is well-defined in $\Ccal_{1/2-4\eps}:= (\nu^\perp \cap B_{1/2-4\eps}) \times (-1-8\eps, 1+8\eps)$. {Moreover, Lemma \ref{l:subsol} guarantees that $v$ is a subsolution of $L^M (\cdot) =0$ away from $\{v=0\}$.} It remains to correct $v_t$ to ensure that it is a subsolution at its free boundary. This is done via a small perturbation by a solution of $L^M(\cdot) = 0$, which, if done in $\Omega^+(v_t)$, may be chosen to be a harmonic function due to our definition of $M$.

First, however, recall that if $u$ is $\eps$-monotone in $B_{1/2}$ in the direction $e\in \Gamma(\theta_0,\nu)$ for $\theta_0 > \frac{\pi}{4}$, then there exists a Lipschitz graph $G$ with Lipschitz constant $\bar L < 1$ over some hyperplane $\pi\subset \R^n$ such that $\{u=0\}$ is contained in the neighborhood $N_\eps(G):= \{\dist(\cdot, G)<\eps\}$, see for instance \cite[Proposition 11.14]{CaffSalsa}. Thus, Lemma \ref{l:1-param-for-barrier} applies. Furthermore, $u$ is fully monotone in $B_{1/2}\setminus N_{\kappa \eps}(G)$ for some $\kappa > 0$ independent of $\eps$ (cf. \cite[Lemma 11.15]{CaffSalsa}).

\begin{lemma}\label{l:subsol-family}
    Suppose that $u$ and $M$ are as in Assumption \ref{a:eps-mon} and suppose that $u$ is $\eps$-monotone in $B_{1/2}$ in a cone $\Gamma(\theta_0,\nu)$ with $\theta_0 \geq \frac{\pi}{4}$. Consider
    \begin{itemize}
        \item[(a)] the family of functions $\{\vphi_t\}$ as in Lemma \ref{l:1-param-for-barrier} and the corresponding $v_t$ from \eqref{e:subsol};
        \item[(b)] a harmonic function $w_t$ {on $N_{C\kappa \eps}(G) \cap \Omega^+(v_t)$} with boundary data 
        \[
            \begin{cases}
                u & \text{on $\partial N_{C\kappa \eps}(G)
                \cap \Omega^+(v_t)$} \\
                0 & \text{otherwise\,,}
            \end{cases}
        \]
        for some large constant $C$, extended by zero;
        \item[(c)] the function $\bar v_t = v_t + \eta w_t$ for $\eta>0$, defined in $\Ccal_{1/2-4\eps}$\,.
    \end{itemize}
    Then, for $\delta$ as in Lemma \ref{l:1-param-for-barrier} and $\sigma$ as above, if $\eta \geq \frac{C\sigma}{\delta}$ and $\frac{\sigma}{\delta}$ is sufficiently small, $\bar v_t$ is a viscosity sub-solution to $L^M(\cdot) = 0$ in $\Ccal_{1/2-4\eps}\cap N_{C\kappa \eps}(G)$.
\end{lemma}
See \cite[Lemma 6.9]{AM}, \cite[Lemma 5.5]{CaffSalsa} for the proof of Lemma \ref{l:subsol-family}. Note that we take $w_t$ to be harmonic due to the fact that we are defining it in $\Omega^+(v_t)$, where the coefficient $M$ is the identity for our operator $L^M$.

{We are now in a position to combine these Lemmas to conclude the validity of Proposition \ref{p:impr-of-mon}. The argument can be found in \cite[Proof of Proposition 6.11]{AM} (see also \cite[Proof of Lemma 5.7]{CaffSalsa}), but we repeat it here to demonstrate how all of the preceding tools fit together.

\begin{proof}[Proof of Proposition \ref{p:impr-of-mon}]
    Let $\nu$ be as in the statement of the proposition. Fix $\lambda\in (0,1)$, to be determined later.
    
    For $\eps \in (0,\eps_1)$, where the latter is also to be determined, consider the function
    \[
        \tilde u (x) := u (x- \lambda \eps \nu)\,.
    \]
    The $\eps$-monotonicity of $u$, together with the observation that
    \[
        B_{\eps(\sin \theta - (1-\lambda)}(x-\lambda \eps \nu)\subset B_{\eps\sin \theta} (x-\eps\nu)\,,
    \]
    guarantees that
    \[
        \sup_{B_{\eps(\sin\theta - (1-\lambda))}(x)} \tilde u \leq u(x)\,,
    \]
    as long as $1-\lambda < \sin(\pi/4)$.

    On the other hand, as pointed out in the discussion preceding Lemma \ref{l:subsol-family}, it is again elementary to check that the $\eps$-monotonicity of $u$ guarantees full monotonicity on $N_{\kappa_0 \eps}:= \{x: \dist(x, \nu^\perp) > \kappa_0\eps\}$ for a suitable positive constant $\kappa_0$. This in particular implies that
    \[
        \sup_{B_{\lambda\eps \sin \theta}(x)} \tilde u \leq u(x)\qquad \text{for any $x\notin N_{\kappa_0\eps}$.}
    \]
    It therefore remains to verify that there exists $\bar{t}\in[0,1]$ such that for every $t\in[0,t]$ we have
    \begin{equation}\label{e:upper-bd-param-fn}
        v_t(x) \leq u(x) \qquad \text{for any $x\in \Ccal_{1/2-c_0\eps^{1/8}}\cap N_{C\kappa_0\eps}$\,,}
    \end{equation}
    for $v_t = \sup_{B_{\sigma\vphi_t(x)}(x)} u$ as in \eqref{e:subsol}, where $\{\vphi_t\}$ is the family given by Lemma \ref{l:1-param-for-barrier}, for a suitable choice of parameters. Choose
    \begin{align*}
        \sigma = \eps(\sin\theta - (1-\lambda))\,, \qquad \lambda \geq \frac{3-\sqrt{2}}{2}\,, \qquad \eta = C\eps^{1/4}\,, \qquad \delta = \eps^{1/2}\,.
    \end{align*}\
    Moreover, since we want to allow room for the correcting harmonic function $\eta w_t = O(\eps^{1/4})$ from Lemma \ref{l:subsol-family}, we will in addition impose
    \begin{equation}\label{e:radius-lb}
        \sigma \vphi_t \leq \eps(\lambda \sin \theta - c\eps^{1/4})\,,
    \end{equation}
    for a suitable constant $c$ to be determined. Invoking the properties of $\vphi_t$ from Lemma \ref{l:1-param-for-barrier}, combined with the above choices of parameters, this rearranges to the condition
    \[
        1+t \leq \frac{\lambda\sin\theta - c\eps^{1/4}}{\sin\theta - (1-\lambda)}\,.
    \]
    Fix $\bar t$ to attain equality above; this may be done by taking $\lambda$ sufficiently close to 1. We henceforth fix $t\in [0,\bar t]$. 
    
    Apply Lemma \ref{l:subsol-family} to obtain the function $\bar v_t = v_t + \eta w_t$ therein. We claim that 
    \begin{equation}\label{e:subsol-main-bd}
        \bar v_t \leq u \qquad \text{in $\Ccal_{1/2-c_0\eps^{1/8}}\cap N_{C\kappa_0\eps}$}\,,
    \end{equation}
    for $c_0$ chosen appropriately. From here, we immediately conclude the desired upper bound \eqref{e:upper-bd-param-fn}, since  construction $w_t$ is non-negative, and $\eta w_t \leq C\eps^{1/4}$.
    
    Towards showing the claim, first of all recall the linear growth (which kicks in at a distance $\simeq \eps$ away from $\Omega^+(u)$) 
    \[
        |\nabla \tilde u(x)|\simeq \frac{\tilde u(x)}{\dist(x,\partial \Omega^+(\tilde u))}\,,
    \]
    the proof of which can be found in \cite[Lemma 5.6]{CaffSalsa}. From this we deduce that for any $0<\rho_1\leq \rho_2\leq \lambda\eps\sin\theta$, on the set $\partial N_{C\kappa_0\eps}\cap \Ccal_{1/2-4\eps}$ we have
    \[
        \sup_{B_{\rho_1}(x)}\tilde u \leq \sup_{B_{\rho_2}(x)}\tilde u - (\rho_2 - \rho_1) |\nabla \tilde u(x)| \leq \left(1-\sgn(u)\frac{\rho_2-\rho_1}{C\kappa_0\eps}\right) u(x)
    \]
    where $\sgn(u)$ denotes the sign of $u$ in $B_{\rho_2}(x)$. Choosing 
    \[
        \rho_1 = \sigma \vphi_t \qquad \text{and} \qquad \rho_2 = \eps(\lambda \sin \theta - \tfrac{c\eps^{1/4}}{2})
    \]
    and recalling that the former satisfies the bound \eqref{e:radius-lb}, we deduce that, as long as in addition $\bar t \in (0, \lambda \sin \theta - c\eps^{1/4})$,
    \[
        \bar v_t \leq u \qquad \text{in} \qquad \partial N_{C\kappa_0\eps}\cap \Ccal_{1/2-4\eps}\,.
    \]
    On the other hand,
    \[
        \bar v_t \leq u \qquad \text{in} \qquad N_{C\kappa_0\eps}\cap \partial\Ccal_{1/2-4\eps}
    \]
    by a simple application of the boundary Harnack inequality, since $\vphi_t \simeq 1$ close to $\partial\Ccal_{1/2-4\eps}$ (see (iii) of Lemma \ref{l:1-param-for-barrier}).

    From here, we conclude the validity of \eqref{e:subsol-main-bd} for a suitable choice of $t$ by a continuity-in-$t$ argument. Namely, letting
    \[
        E:= \{t\in [0,1] :  \bar v_t \leq u \ \text{in} \ N_{C\kappa_0\eps}\cap \Ccal_{1/2-4\eps}\}\,,
    \]
    we note that $0 \in E$, and we claim that $E = [0,\bar t]$. Indeed, if this were not the case, then since $E$ is clearly closed we may consider the first closed interval $I = [0,t_0] \subsetneq E$. Then, since $\bar v_{t_0}$ is a viscosity sub-solution of $L^M(\cdot)= 0$ in $\Ccal_{1/2-4\eps}$ and we have just verified that $\bar v_{t_0} \leq u$ on $\partial (N_{C\kappa_0 \eps} \cap \Ccal_{1/2-4\eps}$, we deduce that the free boundaries of $v_{t_0}$ and $u$ must touch. The latter is not permitted, however, for any pair of functions consisting of a viscosity solution and a viscosity sub-solution to our free boundary problem; see \cite[Lemma 6]{Feldman} and \cite[Lemma 6.7]{AM} (see also \cite[Lemma 4.9]{CaffSalsa}).
\end{proof}}

\subsection{Global Lipschitz boundaries are flat}
With Theorem \ref{t:Lip} at hand, we may now combine with the results of \cite{Feldman} and \cite[Section 7]{AM}, which rely on the techniques of \cite{Caff1-Lip-implies-reg}. The strategy is similar to that from Section \ref{ss:eps-mon-implies-Lip}. Namely, one must construct a family of subsolutions, that, in this setting, provides an improvement of monotonicity from the initial cone given by the conclusion of Theorem \ref{t:Lip}, to a cone with larger angle, which amounts to improving the Lipschitz constant of the graph of $\{u=0\}$, in a slightly smaller cylinder in the domain.

{Since we may now simply apply the main result of \cite{Feldman} (see also \cite[Section 7]{AM} and \cite{Cerutti-Ferrari-Salsa}) as a black box to the function that satisfies the conclusions of Theorem \ref{t:Lip}}, we omit the details here, and simply refer the reader to the aforementioned references, providing merely the final conclusion in the form that we wish to use globally.

\begin{theorem}\label{t:Lip-implies-C1alpha}
    Suppose that $u$ and $M$ satisfy Assumption \ref{a:blownup} and that $\{u=0\}\cap \Ccal_{1/4}$ is the graph of a Lipschitz function with Lipschitz constant $\kappa>0$. Then $\{u=0\}\cap \Ccal_{1/8}$ is $C^{1,\alpha}$ for some $\alpha(n,\kappa, \Lambda_0)>0$.
\end{theorem}

We are now in a position to conclude the key rigidity theorem for globally flat free boundaries (cf. \cite[Lemma 6.2]{DFS-2phase}).

\begin{corollary}\label{c:flat-fb}
    Suppose that $u$, $\omega$, $M$ satisfy Assumption \ref{a:eps-mon} and let $r_0>0$ be fixed. For $\delta_1=\delta_1(n,\Lambda_0,C_2,R_2)>0$ given by Theorem \ref{t:Lip}, if $d_r(\omega,\Fcal)\leq \delta_1$ for every $r\geq r_0$, then $u$ is a two-plane solution of $L^M u=0$. In particular $\omega \in \cF$ and $\{u=0\} \in G(n-1,n)$.
\end{corollary}

\begin{remark}\label{r:TerrSoa}
    In light of Theorem \ref{t:Lip} and Theorem \ref{t:Lip-implies-C1alpha}, as in other two-phase free boundary problems, it is natural to ask about the existence of singularities of $\{u=0\}$, the structure of said singularities, and the behavior of $u$ at such points. In the classical case of $M=\id$, the monotonicity formula of Alt-Caffarelli-Friedman may be used to detect flat portions of the free boundary, characterizing the blow-ups of $u$ as two-plane solutions at such points. Thus singularities are contained in the set where the ACF density is vanishing, thus characterizing the order of vanishing of $u$ as being superlinear at singular points. An analogue of the ACF monotonicity formula was established in the multi-operator setting by Soave-Terracini \cite{TerrSoa}, but this monotonicity formula detects an order of vanishing that is necessarily \emph{sublinear} in the case where the coefficients of the operators have a discontinuity at the boundary. This is supplemented by the existence of an example in $\R^3$ of a pair of complementary solutions to the multi-operator elliptic problem that have the same sublinear decay to zero at the free boundary. However, in that example the zero set  has positive Lebesgue measure. It therefore remains an open question whether such an example can exist under Assumption \ref{a:eps-mon}. Due to Theorems \ref{t:Lip} and \ref{t:Lip-implies-C1alpha}, such an example, if it exists, must have a zero set that is sufficiently far from flat. 
\end{remark}

\begin{proof}[Proof of Corollary \ref{c:flat-fb}]

    {In light of Theorem \ref{t:Lip}, the crucial flatness at infinity assumption, $d_{r}(\omega, \cF) \le \delta_{1}$ for all $r \ge r_{0}$, tells us that, after a rotation the boundary $\{u=0\}\cap \Ccal_{1/4}$, a graph
    \[
        \{u=0\}\cap \Ccal_{1/4} = \{(x',x_n) : x_n = g(x')\}\,,
    \]
    where $g$ is Lipschitz with constant $\kappa>0$ which depends \emph{only} on allowable constants.} In turn applying Theorem \ref{t:Lip-implies-C1alpha}, up to a rotation of coordinates, we deduce that  $\{u=0\}\cap \Ccal_{1/8}$ is $C^{1,\alpha}$ for $\alpha>0$ depending only on $n$, $\kappa$ and $\Lambda_0$. The latter in particular yields the estimate 
    \begin{equation}\label{e:C1alpha-est}
        |g(x') - g(0) - \nabla g(0)\cdot x'| \leq C |x'|^{1+\alpha}\qquad \forall x'\in B_{R/8}\cap e_n^\perp\,,
    \end{equation}
    where $C= C(n,\kappa,\Lambda_0)>0$ and $e_n$ is a unit normal to the hyperplane $\{x_n=0\}$. On the other hand, observe that for any $R>0$, {the rescaled measure $\omega_{0,R}:= T_{0,R}[\omega]$ (whose support is the graph of the rescaled function $g_R(x') := R^{-1} g(Rx')$) also satisfies the flatness at infinity assumption $d_{r}(\omega_{0,R},\cF) \le \delta_{1}$ for all $r \ge r_{0}$. So we again apply Theorems \ref{t:Lip} and \ref{t:Lip-implies-C1alpha} in the rescaled setting and} the estimate \eqref{e:C1alpha-est} for $g_R$ becomes
    \begin{equation}\label{e:C1alpha-estR}
        |g(y') - g(0) - \nabla g(0)\cdot y'| \leq C R^{-\alpha} |y'|^{1+\alpha}\qquad \forall x'\in B_{R/8}\cap e_n^\perp\,.
    \end{equation}
    Taking $R \to +\infty$ in \eqref{e:C1alpha-estR} implies $g$ is{ affine}. This in turn implies $u$ is a two-plane solution. Since $u$ is a two-plane solution $\{u=0\} =: \pi \in G(n-1,n)$. Since $\omega \in \cD(\id,M)$, $\spt \omega = \{u=0\} = \pi$. As a two-plane solution $u$ is translation invariant for $x \in \pi$.  This in turn ensures $\omega$ has constant density on $\pi$, proving $\omega \in \cF$.
\end{proof}

\section{Proof of Theorem \ref{t:decomp}}\label{s:Preiss}

Recalling the strategy from Section \ref{s:overview}, we wish to apply Lemma \ref{l:connectedness} to the dilation cone $\cD(A_{1},A_{2})$. With this in mind, we first require the following lemma.

\begin{proposition} \label{p:compactbasis}
    Let $A^\pm \in \R^{n\times n}$ be a pair of elliptic matrices. The dilation cone $\cD(A^+,A^-)$ has a compact basis.
\end{proposition}

\begin{proof}
    Suppose $\{\omega_{k}\}$ {is a sequence contained in the basis of $\cD(A^{+},A^{-})$, see Definition \ref{d:dcone}}. Then by definition of the basis of a $d$-cone, $F_{1}(\omega_{k}) = 1$ for all $k$. In particular, for all $k$:
    $$
        \frac{1}{2} \omega_{k}(B_{1/2}) \le \int_{B_{1}} (1-|x|) d \omega_{k} = F_{1}(\omega_{k}) = 1.
    $$
    So, by Lemma \ref{l:doubling}, $\{\omega_{k}\}$ is uniformly locally bounded and hence pre-compact with respect to the weak-$*$ topology. So, there exists some Radon measure $\omega$ and a subsequence (which we do not relabel) so that $\omega_{k} \xrightharpoonup{*} \omega$ {and $F_{1}(\omega) = 1$. We need to show that $\omega \in \cD(A^{+},A^{-})$.}

    Moreover, for each $k$, there are Green's functions $u_{k}^{\pm}$ with pole at infinity associated to $(\Omega_{k}^{\pm}, \omega_{k}, L_{A^\pm})$ and $\Omega_{k}^\pm$ are NTA with the same constants, independently of $k$. By {Theorem \ref{t:diagonal-cptness}}, we know that there exist NTA domains $\Omega^{\pm}$ with the same NTA constants and a subsequence which we do not relabel so that $\Omega^{\pm}_{k} \to \Omega$ and $\partial \Omega_{k} \to \partial \Omega$ locally in Hausdorff distance. Moreover, $\spt \omega = \partial \Omega$.
    
    Now, almost exactly as in the proof of Lemma \ref{l:twoplane}, we can show that the sequence of functions $\{u_{k}\}$ for $u_{k} = u_{k}^{+} - u_{k}^{-}$ are uniformly bounded in $W^{1,2}(B_{N})$ for all $N \in \N$. The only difference to that proof, is that we use the uniform local boundedness of $\omega_{k}$ and Lemma \ref{l:cfms} to prove that $\|u_{k}\|_{L^{\infty}(B_{N})}$ is bounded independently of $k$, instead of using the proximity to a fixed two-plane solution. This in turn bounds the $L^{2}$-norms uniformly and allows us to use the Caccioppoli inequality to deduce that the gradients are also uniformly bounded in $L^{2}$. In particular, there is a $W^{1,2}_{\loc}(\R^{n})$ function $u$ so that (up to a subsequence that we do not relabel) $u_{k} \to u$ in $L^{2}_{\loc}(\R^{n})$ and $\nabla u_{k} \rightharpoonup \nabla u$ weakly in $L^{2}_{\loc}(\R^{n})$. The weak convergence of the gradients combined with the local convergence in Hausdorff distance of the $\Omega_{k}^{\pm}$ and $\partial \Omega_{k}$ guarantees that $u$ is a Green's function with pole at infinity for $(\Omega^{\pm}, \omega, L_{A^\pm})$. Indeed, for all $\varphi \in C^{\infty}_{c}(\R^{n})$, writing $u = u^{+} - u^{-}$:

    \begin{align*}
        \int_{\partial \Omega} \varphi \, d \omega & = \lim_{k \to \infty} \int_{\partial \Omega_{k}} \varphi \, d \omega_{k} = \lim_{k \to \infty} \int_{\Omega_{k}^{\pm}} \Langle A^\pm \nabla u_{k}^{\pm}, \nabla \varphi \Rangle d x \\
        & = \int_{\Omega^{\pm}} \Langle A^\pm \nabla u, \nabla \varphi \Rangle dx.
    \end{align*}

    That is, $u^{\pm}$ are Green's functions with pole at infinity for $(\Omega^{\pm}, \omega, L_{A^\pm})$ and $\Omega$ is NTA with the same NTA constants, that is $\omega \in \cD(A^+,A^-)$, confirming that $\cD(A^+,A^-)$ is a compact dilation cone.
\end{proof}

We are now in a position to prove Theorem \ref{t:decomp}.

\begin{proof}[Proof of Theorem \ref{t:decomp}]
    Recall the sets $F_{0}, \dots, F_{4}$ from \eqref{e:gamma1} - \eqref{e:gamma0}, {namely
    \begin{align*}
    	F_0 & = \bigg\{ p \in F_1 \cap \Ccal\Big(\omega^+,\frac{d \omega^{-}}{d \omega^{+}}\Big) \cap \Ccal(\omega^\pm, A^\pm) : \theta(\omega^\pm,p,F_1) \text{ exists and equals 1} \bigg\}; \\
    	F_1 &= \left \{ p \in \partial \Omega : 0 < h(p) := \frac{d \omega^{-}}{d \omega^{+}}(p) = \lim_{r \to 0} \frac{\omega^{-}(B(p,r))}{\omega^{+}(B(p,r))} < \infty \right\}; \\
    	F_2 &= \left\{ p \in \partial\Omega: \frac{d \omega^{-}}{d \omega^{+}}(p) = \infty \right\}; \\
    	F_3 &= \left \{ p \in \partial\Omega : \frac{d \omega^{-}}{d \omega^{+}}(p) = 0 \right\}; \\
    	F_4 &= \left \{ p \in \partial\Omega : \frac{d \omega^{-}}{d \omega^{+}}(p) \text{ does not exist } \right\}\,.
    \end{align*}
	}
    
    We prove the theorem holds when $F^{*}$ is the subset of $F_{0}$ where the ``tangents to tangents are tangents" property holds (i.e. the conclusion of Theorem \ref{t:tan2ltan}). {Note that
    	\[
    	F^* \subset F_0 \subset F_1\,.
    	\]}
    	Let $S = F_{2} \cup F_{3}$, and $N = F_{4} \cup (F_{1} \setminus F^{*})$. By Remark \ref{r:decomp-Leb-diff} it follows that $\partial \Omega = F^{*} \sqcup S \sqcup N$. Similarly, the conclusion (ii) of the theorem holds by Remark \ref{r:decomp-Leb-diff} and conclusion (iii), i.e. $\omega^\pm(N) = 0$, holds by combining Remark \ref{r:decomp-Leb-diff} with Theorem \ref{t:tan2ltan}. It remains to check conclusion (i). The claim of mutual absolute continuity in (i) follows again by Remark \ref{r:decomp-Leb-diff}. To this end, we now show that if $p \in F^{*}$ then $\Tan_{\Lambda}(\omega^{\pm},p) \subset \cF$, which by Lemma \ref{l:iso} in turn implies $\Tan(\omega^{\pm},p) \subset \cF$.
    
    For $p \in F^{*}$, Corollary \ref{c:Lambda-tangent-FB} ensures
    \[
        \Tan_\Lambda(\omega^\pm, p) \subset \Dcal(\Id, M_p)\,,
    \]
    where $M_{p} = \Lambda(p)^{-1} A^{-}(p) \Lambda(p)^{-1}$ for $\Lambda(p) = \sqrt{A^+(p)}$. From Proposition \ref{p:compactbasis}, we know that $\cM=\cD(\id,M_{p})$ is a compact dilation cone. Thus combining Corollary \ref{c:flat-fb} and Lemma \ref{l:connectedness}, we deduce that either $\Tan_\Lambda(\omega^\pm,p) \subset \Fcal$ or $\Tan_\Lambda(\omega^\pm,p) \cap \Fcal = \emptyset$, where $\Fcal$ is as in \eqref{e:flat-meas}. 

    It therefore remains to rule out the possibility that $\Tan_\Lambda(\omega^\pm,p) \cap \Fcal = \emptyset$. 
    Indeed, for any $\omega \in \cD(\id,M_p)$ we know that there exists a ball $B \subset \Omega^{+}(u)$ with $\partial B \cap \partial \Omega = \{p_{0}\}$ at some point $p_0\in \{u=0\}$. Then, by Lemma \ref{l:twoplane}, $\Tan(\omega,p_{0}) \subset \cF$. On the other hand, since $p\in F^*$, Theorem \ref{t:tan2ltan} ensures that $\Tan(\omega,p_{0}) \subset \Tan_{\Lambda}(\omega^{\pm},p)$, thus implying $\Tan_{\Lambda}(\omega^{\pm},p) \cap \cF \neq \emptyset$ as desired.

Finally, we show $\dim_{\cH}(F^{*}) \le n-1$. We wish to use Lemma \ref{l:beta-decay-dim}. Recalling the definition of $\Theta_{\partial\Omega}$ and $\beta_{F^{*}}$ from \eqref{e:Theta} and \eqref{e:beta}, we claim that $\lim_{r \to 0} \Theta_{\partial \Omega}(p,r) = 0$ for all $p \in F^{*}$. Indeed, suppose not. Then there exists $p \in F^{*}$, $\delta > 0$, and a sequence $r_{i} \to 0$ so that
    \begin{equation} \label{e:ca1}
        \Theta_{\partial \Omega}(p,r_{i}) \ge \delta \qquad \forall i.
    \end{equation}
    By Theorem \ref{t:blowups} we know there exists a subsequence so that $\partial \Omega_{i}$ as in \eqref{e:domainseq} and $\omega^{\pm}_{i}$ as in \eqref{e:measureseq} converge simultaneously to some $\partial \Omega_{\infty}$ and $\omega_{\infty}$ respectively (in local Hausdorff distance and in the weak-$*$ topology respectively). Moreover, we know that $\spt \omega_{\infty} = \partial \Omega_{\infty}$. We already showed that since $p \in F^{*}$, it follows $\omega_{\infty} \in \cF$ so that $\spt \omega_{\infty} = \partial \Omega_{\infty} \in G(n-1,n)$. But starting from \eqref{e:ca1} this implies 
    $$
        \delta \le \Theta_{\partial \Omega}(p, r_{i}) = \Theta_{\partial \Omega_{i}}(0,1) \xrightarrow{i \to \infty} \Theta_{\partial \Omega_{\infty}}(0,1) = 0,
    $$
    yielding a contradiction.
    In particular, since $F^{*} \subset \partial \Omega$, it follows that for all $p \in F^{*}$ we have
    \begin{equation*} 
    0 \le \limsup_{r \to 0} \beta_{F^{*}}(p,r) \le \lim_{r \to 0} \Theta_{\partial \Omega}(p,r) = 0.
    \end{equation*}
    Now Lemma \ref{l:beta-decay-dim} implies $\dim_{\cH}(F^{*}) \le n-1$, verifying the final piece of (i).
\end{proof}

\bibliographystyle{alpha}
\bibdata{references}
\bibliography{references}

\end{document}